\documentclass[twoside,11pt]{article}

\usepackage{amsmath,amssymb,amsthm,amsfonts}
\usepackage{appendix}
\usepackage{hyperref}

\usepackage{color}

\newtheorem{lemma}{Lemma}[section]
\newtheorem{theorem}{Theorem}[section]

\newtheorem{proposition}{Proposition}[section]
\newtheorem{remark}{Remark}[section]

\numberwithin{equation}{section}

\newcommand{\dis}{\displaystyle}

\newcommand{\rmre}{{\rm Re}}

\newcommand{\R}{\mathbb{R}}

\newcommand{\Z}{\mathbb{Z}}

\newcommand{\T}{\mathbb{T}}

\newcommand{\FB}{\mathbf{B}}

\newcommand{\FM}{\mathbf{M}}
\newcommand{\FP}{\mathbf{P}}
\newcommand{\FL}{\mathbf{L}}
\newcommand{\FI}{\mathbf{I}}

\newcommand{\CD}{\mathcal{D}}
\newcommand{\CE}{\mathcal{E}}
\newcommand{\CF}{\mathcal{F}}

\newcommand{\CN}{\mathcal{N}}

\newcommand{\na}{\nabla}

\newcommand{\al}{\alpha}
\newcommand{\be}{\beta}
\newcommand{\ga}{\gamma}
\newcommand{\om}{\omega}
\newcommand{\la}{\lambda}
\newcommand{\de}{\delta}
\newcommand{\si}{\sigma}
\newcommand{\pa}{\partial}

\newcommand{\eps}{\epsilon}

\newcommand{\De}{\Delta}
\newcommand{\Ga}{\Gamma}

\begin{document}

\title{\bf  Optimal Time Decay of the Vlasov-Poisson-Boltzmann System in $\R^3$}
\author{\textsc{Renjun Duan} \\[1mm]
\small\it Johann Radon Institute for Computational and Applied Mathematics\\
\small\it Austrian Academy of Sciences\\
\small\it Altenbergerstrasse 69, A-4040 Linz, Austria\\[2mm]
\textsc{Robert M. Strain}\\[1mm]
\small\it Department of Mathematics, Princeton University\\
\small\it Fine Hall, Washington Road, Princeton NJ 08544-1000, USA}
\date{}
\maketitle

\begin{abstract}
The Vlasov-Poisson-Boltzmann System governs the time evolution of
the distribution function for the dilute charged particles in the
presence of a self-consistent electric potential force through the
Poisson equation. In this paper, we are concerned with the rate of
convergence of solutions to equilibrium for this system over $\R^3$.
It is shown that the electric field which is indeed responsible for
the lowest-order part in the energy space reduces the speed of
convergence and hence the dispersion of this system over the full
space is slower than that of the Boltzmann equation without forces,
where the exact difference between both power indices in the
algebraic rates of convergence is $1/4$. For the proof, in the
linearized case with a given non-homogeneous source,  Fourier
analysis is employed to obtain time-decay properties of the solution
operator. In the nonlinear case, the combination of the linearized
results and the nonlinear energy estimates with the help of the
proper Lyapunov-type inequalities leads to the optimal time-decay
rate of perturbed solutions under some conditions on initial data.
\end{abstract}




\tableofcontents

\section{Introduction}

The Vlasov-Poisson-Boltzmann (called VPB in the sequel for
simplicity) system is a physical model describing the time evolution
of dilute charged particles (e.g., electrons) under a given external
magnetic field \cite{MRS,CIP-Book}. The VPB system for one-species
of particles in the whole space $\R^3$ reads
\begin{eqnarray}
 &\dis \pa_t f+\xi\cdot \na_x f+\na_x\Phi\cdot\na_\xi f = Q(f,f),\label{VPB.1}\\
&\dis  \De_x \Phi=\int_{\R^3}fd\xi-\bar{\rho}(x),\label{VPB.2}
\end{eqnarray}
with initial data
\begin{equation}\label{VPB.ID}
    f(0,x,\xi)=f_0(x,\xi).
\end{equation}
Here, the unknown $f=f(t,x,\xi)$ is a non-negative function standing
for the number density of gas particles which have position
$x=(x_1,x_2,x_3)\in\R^3$ and velocity
$\xi=(\xi_1,\xi_2,\xi_3)\in\R^3$ at time $t>0$. $Q$ is the bilinear
collision operator for the hard-sphere model defined by
\begin{eqnarray*}
  \nonumber&\dis Q(f,g)=\int_{\R^3\times S^{2}}(f'g_\ast'-fg_\ast)
  |(\xi-\xi_\ast)\cdot\om |d\om d\xi_\ast, \\
 \nonumber&\dis f=f(t,x,\xi),\ \ f'=f(t,x,\xi'), \ \ g_\ast=g(t,x,\xi_\ast), \
\ g_\ast'=g(t,x,\xi_\ast'),\\
\nonumber&\dis \xi'=\xi-[(\xi-\xi_\ast)\cdot \om]\om,\ \
\xi_\ast'=\xi_\ast+[(\xi-\xi_\ast)\cdot \om]\om,\ \ \om\in S^{2}.
\end{eqnarray*}
The potential function $\Phi=\Phi(t,x)$ generating the
self-consistent electric field in \eqref{VPB.1} is coupled with
$f(t,x,\xi)$ through the Poisson equation \eqref{VPB.2}.
$\bar{\rho}(x)$ denotes the stationary background density satisfying
\begin{equation*}
    \bar{\rho}(x)\to\rho_\infty\ \ \text{as}\ \ |x|\to\infty,
\end{equation*}
for  a positive constant state $\rho_\infty>0$.

The existence of stationary solutions to the system
\eqref{VPB.1}-\eqref{VPB.2} and the nonlinear stability of solutions
to the Cauchy problem \eqref{VPB.1}-\eqref{VPB.ID} near the
stationary state were obtained in \cite{DY-09VPB}, and the
corresponding results have been recalled in the appendix. In this
paper, we are concerned with the rate of convergence of solutions
towards the stationary states. Since the background density does not
produce any essential difficulty, for simplicity it is supposed
throughout this paper that
\begin{equation}\label{ass.backg}
 \bar{\rho}(x)\equiv \rho_\infty=1, \ \ x\in\R^3.
\end{equation}
In this case,  the VPB system \eqref{VPB.1}-\eqref{VPB.2} has a
stationary solution $(f_\ast,\Phi_\ast)$ with
\begin{equation*}
    f_\ast=\FM,\ \ \Phi_\ast=0,
\end{equation*}
where
\begin{equation*}
\FM=\frac{1}{(2\pi)^{ 3/2}}\exp\left(-|\xi|^2/2\right)
\end{equation*}
is  a normalized global Maxwellian in three dimensions which has
zero bulk velocity and unit density and temperature. One of the main
results of this paper is stated as follows. Notations and norms will be
explained at the end of this section.

\begin{theorem}\label{thm.cr}
Let $N\geq 4$ and $w(\xi)=(1+|\xi|^2)^{1/2}$. Assume  that $f_0\geq
0$ and
\begin{equation}\label{thm.cr.1}
    \left\|\frac{f_0-\FM}{\sqrt{\FM}}\right\|_{H^N\cap L^2_w\cap Z_1}
\end{equation}
is sufficiently small. Let $f\geq 0$ be the solution obtained in
Proposition \ref{a.prop.exi} to the Cauchy problem \eqref{VPB.1},
\eqref{VPB.2} and \eqref{VPB.ID} under the assumption
\eqref{ass.backg}. Then, $f$ enjoys the estimate with algebraic rate
of convergence:
\begin{equation}\label{thm.cr.2}
 \left\|\frac{f(t)-\FM}{\sqrt{\FM}}\right\|_{H^N}\leq C
 \left\|\frac{f_0-\FM}{\sqrt{\FM}}\right\|_{H^N\cap Z_1}
  (1+t)^{-\frac{1}{4}}.
\end{equation}
Furthermore, under the following additional conditions on $f_0$, $f$
also enjoys some estimates with extra rates of convergence:

\medskip

\noindent{Case 1.} If
\begin{equation}\label{thm.cr.3}
    \int_{\R^3}(f_0(x,\xi)-\FM)d\xi\equiv 0
\end{equation}
holds for any $x\in\R^3$, then one has
\begin{equation}\label{thm.cr.4}
 \left\|\frac{f(t)-\FM}{\sqrt{\FM}}\right\|_{H^N}\leq C
 \left\|\frac{f_0-\FM}{\sqrt{\FM}}\right\|_{H^N\cap Z_1}
  (1+t)^{-\frac{3}{4}}.
\end{equation}

\medskip

\noindent{Case 2.} Fix any $0<\eps\leq 3/4$ and suppose that
\begin{equation}\label{thm.cr.5}
    \left\|\frac{f_0-\FM}{\sqrt{\FM}}\right\|_{H^N_w\cap Z_1}
\end{equation}
is sufficiently small.
Then one has
\begin{eqnarray}
 &&\dis  \left\|\frac{f(t)-\sum_{j=0}^4\langle e_j,f(t)\rangle e_j\FM}{\sqrt{\FM}}\right\|_{
 H^N_w}\nonumber\\
 &&\dis\hspace{3cm} + \|\langle e_0,f(t)-\FM\rangle\|_{H^N_x}+\sum_{j=1}^4
 \|\na_x \langle e_j,f(t)\rangle\|_{H^{N-1}_x}\nonumber \\
 &&\dis \leq C\left\|\frac{f_0-\FM}{\sqrt{\FM}}\right\|_{ H^N_w\cap
    Z_1} (1+t)^{-\frac{3}{4}+\eps},\label{thm.cr.7}
\end{eqnarray}
where $\{e_j\}_{j=0}^4$ is the orthonormal set in $L^2(\R^3;\FM
d\xi)$ defined by
\begin{equation}
  e_0=1,\ \
   e_j=\xi_j\,(1\leq j\leq 3),\ \
   e_4 =\frac{|\xi|^2-3}{\sqrt{6}}. \label{thm.cr.8}
\end{equation}
\end{theorem}

\begin{remark}
The rates of convergence in \eqref{thm.cr.2} and \eqref{thm.cr.4}
are optimal under the corresponding assumptions in the sense that
they coincide with those rates given in Theorem \ref{thm.liDe} at
the level of linearization. Similarly, the rate of convergence in
\eqref{thm.cr.7} is almost optimal.
To our
knowledge, these results are the first ones in the study of rate of
convergence for the nonlinear VPB system in $\R^3$. Of course, it is
an interesting problem to improve the almost optimal rate in
\eqref{thm.cr.7} to the optimal rate.
\end{remark}

\begin{remark}\label{rem.cr.1}
In \cite{YZ-CMP} on the same issue, some slower algebraic rates of
convergence of solutions in the space $L^\infty_x(L^2_\xi)$ were
obtained on the basis of the pure energy estimates and time-decay of
some ordinary differential equation. Here, under the assumptions of
\eqref{thm.cr.5}, from the Sobolev inequality,  \eqref{thm.cr.7}
shows that
\begin{equation*}
    \left\|\frac{f(t)-\FM}{\sqrt{\FM}}\right\|_{L^\infty_x(H^{N-2}_{\xi,w})}\leq
    C\left\|\frac{f_0-\FM}{\sqrt{\FM}}\right\|_{ H^N_w\cap
    Z_1} (1+t)^{-\frac{3}{4}+\eps},
\end{equation*}
where
$L^\infty_x(H^{N-2}_{\xi,w})=L^\infty(\R^3_x;H^{N-2}(\R^3_\xi;wd\xi))$.
The above inequality implies the almost optimal rate of convergence
of solutions. The main reason why one can here obtain the optimal or
almost optimal rates is that we make full use of the time-decay
properties of solutions to the linearized system with nonhomogeneous
sources. This will be carried out in Section \ref{sec.ld}, where we
shall also state another main result Theorem \ref{thm.liDe} in this
paper.
\end{remark}

\begin{remark}
As shown in \cite{UY-AA}, for the Boltzmann equation without
external forces, the power index in \eqref{thm.cr.2} takes the value
$3/4$. Hence, the dispersion of the VPB system in $\R^3$ is slower
than that of the Boltzmann equation. This is essentially caused by
the self-induced potential force. Furthermore, if one decomposes $f$
as the summation of three parts
\begin{equation*}
    f=\langle e_0, f\rangle e_0 \FM +\sum_{j=1}^4 \langle
    e_j,f\rangle e_j\FM +\left\{f- \sum_{j=0}^4 \langle
    e_j,f\rangle e_j\FM \right\},
\end{equation*}
then by comparing \eqref{thm.cr.2} and  \eqref{thm.cr.7}, one can
find out that the effect of the  self-induced potential force on
rates of convergence only happens in the above second part which
corresponds to the projections of $f$ along the momentum components
$e_j$ $(1\leq j\leq 3)$ and the temperature component $e_4$. This
phenomenon is consistent with that recently obtained by
\cite{LMZ-NSP} for the Navier-Stokes-Poisson system in $\R^3$.

\end{remark}

The time rate of convergence to equilibrium is an important topic in the
mathematical theory of the physical world. As pointed out in
\cite{Vi}, there exist general structures in which the interaction
between a conservative part and a degenerate dissipative part lead
to convergence to equilibrium, where this property was called {\it
hypocoercivity}. Theorem \ref{thm.cr} indeed provides a concrete
example of the hypocoercivity property for the nonlinear VPB system
in the framework of perturbations. The key of the method to study
hypocoercivity provided by this paper is to carefully capture the
time-decay rates for the perturbed macroscopic system of equations
with the hyperbolic-parabolic structure, which is in the same spirit
of the Kawashima's work \cite{Ka}.

There has been extensive investigations on the rate of convergence
for the nonlinear Boltzmann equation or related spatially
non-homogeneous kinetic equations with relaxations. In what follows
let us mention some of them. In the context of perturbed solutions,
the first result was given by Ukai \cite{Ukai-1974}, where the
spectral analysis was used to obtain the exponential rates for the
Boltzmann equation with hard potentials on torus. The results in
\cite{Ukai-1974} were improved by Ukai-Yang \cite{UY-AA} in order to
consider existence of time-periodic states in the presence of
time-periodic sources, which was later extended by
Duan-Ukai-Yang-Zhao \cite{DUYZ} to the case with time-periodic
external forcing by using the  energy-spectrum method. We also
mention Glassey-Strauss \cite{GS-DCDS} for the study of the
essential spectra of the solution operator of the VPB system.
Recently, Strain-Guo \cite{SG} developed a weighted energy method to
get the exponential rate of convergence for the Boltzmann equation
and Landau equation with soft potentials on the torus.  Earlier but
along the same line of research, Strain-Guo  \cite{SG0} developed a
general theory of polynomial decay rates to any order in a unified
framework and applied it to four kinetic equations, the
Vlasov-Maxwell-Boltzmann System, the relativistic Landau-Maxwell
System, the Boltzmann equation with cutoff soft-potentials and the
Landau equation all on the torus.

Another powerful tool is entropy method which has  general
applications in the existence theory for nonlinear equations. By
using this method as well as the elaborate analysis of  functional
inequalities, time-derivative estimates  and interpolation,
Desvillettes-Villani \cite{DV} obtained first the almost exponential
rate of convergence of solutions to the Boltzmann equation on torus
with soft potentials for large initial data under the additional
regularity conditions that all the moments of $f$ are uniformly
bounded in time and $f$ is bounded in all Sobolev spaces uniformly
in time. See Villani \cite{Vi} for extension and simplification of
results in \cite{DV} still conditionally to smoothness bounds by
further designing a new auxiliary functional. Notice that \cite{SG0}
provided a simple proof of \cite{DV} for the unconditional
perturbative regime. Recently, by finding some proper Lyapunov
functional defined over the Hilbert space, Mouhot-Neumann \cite{MN}
obtained the exponential rates of convergence for some kinetic
models with general structures in the case of torus; see also
\cite{Vi} for the  similar study. An extension of \cite{MN} to
models with additional confining potential forces was given by
Dolbeault-Mouhot-Schmeiser \cite{DMS}.

Besides those methods mentioned above for the study of rates of
convergence, the method of Green's functions was also founded by
Liu-Yu \cite{Liu-Yu-1DG} to expose the pointwise large-time behavior
of solutions to the Boltzmann equation in the full space $\R^3$.

Here, we mention that if there is no collision in
\eqref{VPB.1}-\eqref{VPB.2}, that is to consider the Vlasov-Poisson
system, then the so-called Landau damping comes out. This was
recently studied by Mouhot-Villani \cite{MV} on torus, where it was
shown that even though in the absence of kinetic relaxation, in the
analytic regime the solution still converges weakly in some sense to
certain large-time states determined by the initial data and the
nonlinear system itself, for any interaction potential less singular
than Coulomb.  In the Coulomb case, they established Landau damping
over exponentially long times.

Finally, we also mention some of results on the existence theory of
the VPB system and related kinetic equations: global existence of
renormalized solutions with large initial data
\cite{DL.BE,DL-CPAM,Mischler}, global existence of classical
solutions near Maxwellians \cite{Liu-Yu-Shock,Liu-Yang-Yu,YYZ-ARMA},
\cite{Guo2,Guo-IUMJ,S.VMB} and \cite{Duan,DY-09VPB}, and global
existence of  solutions near vacuum \cite{Guo3,DYZ-DCDS}.

The rest of this paper is organized as follows. In Section
\ref{sec.MMD} we make some preparations to reformulate the Cauchy
problem of the VPB system, make the macro-micro decomposition for
both the solution and equations, obtain a system of equations
describing the evolution of some macroscopic velocity moment
functions for the later analysis in both the linear and nonlinear
cases, and reduce Theorem \ref{thm.cr} in an equivalent form to
Proposition \ref{prop.main}. In Section \ref{sec.ld} and Section
\ref{sec.ES}, we obtain the time-decay rates of perturbed solutions
under some conditions on initial data in the linear and nonlinear
cases, respectively. Here, Theorem \ref{thm.liDe} which is another
main result of this paper is applied together with some energy
estimates to prove Proposition \ref{prop.main}. Some free energy
functionals and Lyapunov-type inequalities play a key role in the
proof of  Theorem \ref{thm.liDe} and  Proposition \ref{prop.main}.
Finally, we conclude this paper with an appendix in Section
\ref{sec.app} by listing some results obtained in \cite{DY-09VPB}
about the existence of the stationary solution and its nonlinear
stability for the VPB system \eqref{VPB.1}-\eqref{VPB.2}.

\medskip

\noindent{\bf Notations.}\ Throughout this paper,  $C$  denotes
some positive (generally large) constant and $\la$ denotes some positive (generally small) constant, where both $C$ and
$\la$ may take different values in different places. In addition,
$A\sim B$ means $\la_1 A\leq B \leq \la_2 A$ for two generic
constants $\la_1>0$ and $\la_2>0$. For an integrable function $g:
\R^n\to\R$, its Fourier transform $\widehat{g}=\CF g$ is defined by
\begin{equation*}
   \widehat{g}(k)= \CF g(k)= \int_{\R^n} e^{-2\pi i x\cdot k} g(x)dx, \ \ x\cdot
    k=:\sum_{j=1}^nx_jk_j,
\end{equation*}
for $k\in\R^n$, where $i =\sqrt{-1}\in \mathbb{C}$ is the imaginary
unit. For two complex vectors $a,b\in\mathbb{C}^n$, $(a\mid
b)=a\cdot \overline{b}$ denotes the dot product over the complex
field, where $\overline{b}$ is the complex conjugate of $b$. For any
integer $m\geq 0$, we use $H^m$, $H^m_x$, $H^m_\xi$ to denote the
usual Hilbert spaces $H^m(\R^n_x\times\R^n_\xi)$, $H^m(\R^n_x)$,
$H^m(\R^n_\xi)$, respectively, where $L^2$, $L^2_x$, $L^2_\xi$ are
used for the case when $m=0$. For a Banach space $X$,
$\|\cdot\|_{X}$ denotes the corresponding norm, while $\|\cdot\|$
always denotes the norm $\|\cdot\|_{L^2}$ or $\|\cdot\|_{L^2_x}$ for
simplicity. We also use $\|\cdot\|_{H^m_w}$ for the norm of the
weighted Hilbert space $H^m(\R^n_x\times\R^n_\xi;w(\xi)dxd\xi)$ and
$\|\cdot\|_{L^2_w}$ for $L^2(\R^n_x\times\R^n_\xi;w(\xi)dxd\xi)$.
We use $\langle\cdot,\cdot\rangle$ to denote the inner product over
the Hilbert space $L^2_\xi$, i.e.
\begin{equation*}
    \langle g,h\rangle=\int_{\R^n} g(\xi)h(\xi)d\xi,\ \ g,h\in
    L^2_\xi.
\end{equation*}
For $q\geq 1$, we also define
\begin{equation*}
    Z_q=L^2_\xi(L^q_x)=L^2(\R^n_\xi;L^q(\R^n_x)),\ \ \|g\|_{Z_q}=\left(\int_{\R^n}\left(\int_{\R^n}
    |g(x,\xi)|^qdx\right)^{2/q}d\xi\right)^{1/2}.
\end{equation*}
For the multiple indices $\al=(\al_1,\cdots,\al_n)$ and
$\be=(\be_1,\cdots,\be_n)$, as usual we denote
\begin{equation*}
 \pa_x^{\al}\pa_\xi^\be=\pa_{x_1}^{\al_1}\cdots\pa_{x_n}^{\al_n}
    \pa_{\xi_1}^{\be_1}\cdots\pa_{\xi_n}^{\be_n}.
\end{equation*}
The length of $\al$ is $|\al|=\al_1+\cdots+\al_n$. For simplicity,
we also use $\pa_j$ to denote $\pa_{x_j}$ for each $j=1,\cdots,n$.
Generally, except in Section \ref{sec.ld}, we consider the case of $n=3$ dimensions.

\section{Macro-micro decomposition}\label{sec.MMD}
In this section, we shall make some preparations for the later
analysis in both the linear and nonlinear cases, and moreover reduce
Theorem \ref{thm.cr} to Proposition \ref{prop.main} in the
equivalent form. Firstly, one can reformulate the Cauchy problem
\eqref{VPB.1}, \eqref{VPB.2} and \eqref{VPB.ID} as follows. Set the
perturbation $u=u(t,x,\xi)$ by
\begin{equation}\label{def.per.f}
f=\FM+\sqrt{\FM}u.
\end{equation}
Then $u$ and $\Phi$ satisfy the perturbed system:
\begin{eqnarray}
&\dis \pa_t u+\xi\cdot\na_x u+\na_x \Phi\cdot\na_\xi
u-\frac{1}{2}\xi\cdot\na_x \Phi u -\na_x\Phi \cdot\xi \sqrt{\FM}
=\FL u+\Ga(u,u),\label{VPB.p1}\\
&\dis \De_x\Phi=\int_{\R^3} \sqrt{\FM}ud\xi,\label{VPB.p2}
\end{eqnarray}
with given initial data
\begin{equation}\label{VPB.pID}
    u(0,x,\xi)=u_0(x,\xi)\equiv \frac{f_0-\FM}{\sqrt{\FM}},
\end{equation}
where $\FL u$ and $\Ga(u,u)$ are denoted by
\begin{eqnarray}
  \FL u &=& \frac{1}{\sqrt{\FM}}\left[Q(\FM,\sqrt{\FM}u)+Q(\sqrt{\FM}u,\FM)\right],\label{def.L} \\
  \Ga(u,u)&=&
  \frac{1}{\sqrt{\FM}}Q(\sqrt{\FM}u,\sqrt{\FM}u).\label{def.Ga}
\end{eqnarray}

It is well-known that for the linearized collision operator $\FL$,
one has
\begin{eqnarray*}
  (\FL u)(\xi) &=& -\nu(\xi)u(\xi)+(Ku)(\xi),\\
   \nu(\xi)&=&\int_{\R^3\times
  S^{2}}|(\xi-\xi_\ast)\cdot\om|\FM_\ast\,d\om
  d\xi_\ast,\\
\nonumber (Ku)(\xi)&=&\int_{\R^3\times S^{2}}
\left(-\sqrt{\FM}u_\ast+\sqrt{\FM_\ast'}u'+\sqrt{\FM'}u_\ast'\right)|(\xi-\xi_\ast)\cdot\om|\sqrt{\FM_\ast}d\om
  d\xi_\ast\\
  &=&\int_{\R^3}K(\xi,\xi_\ast)u(\xi_\ast)d\xi_\ast,
\end{eqnarray*}
where $\nu(\xi)$ is called the collision frequency and $K$ is a
self-adjoint compact operator on $L^2(\R^3_\xi)$ with a real
symmetric integral kernel $K(\xi,\xi_\ast)$.  The null space of the
operator $\FL$ is the five dimensional space spanned by the
collision invariants
\begin{equation*}
    \CN=Ker \FL={\rm span}\left\{\sqrt{\FM};\,\xi_i\sqrt{\FM},i=1,2,3;\,
    |\xi|^2\sqrt{\FM}\right\}.
\end{equation*}
From the Boltzmann's H-theorem, the linearized collision operator
$\FL$ is non-positive and moreover, $-\FL$ is locally coercive in
the sense that there is a constant $\la>0$ such that
\begin{equation*}
    -\int_{\R^3}u\FL u\,d\xi\geq \la\int_{\R^3}\nu(\xi)\left(\{\FI-\FP\}u\right)^2d\xi,\ \
    \forall\,u\in D(\FL),
\end{equation*}
where for fixed $(t,x)$, $\FP$ denotes the projection operator from
$L^2_\xi$ to $\CN$ and $D(\FL)$ is the domain of $\FL$ given by
\begin{equation*}
    D(\FL)=\Big\{u\in L^2_\xi\ \Big|\ \sqrt{\nu(\xi)}u\in  L^2_\xi\Big\}.
\end{equation*}
In addition, for the hard sphere model, $\nu$
satisfies
\begin{equation*}
    \nu(\xi)\sim (1+|\xi|^2)^{\frac{1}{2}}=w(\xi).
\end{equation*}
This property will be used throughout this paper.

Given any $u(t,x,\xi)$, we define the projection operator, $\FP u$, as
\begin{equation}\label{form.P}
    \FP
    u=\left\{a^u(t,x)+\sum_{j=1}^3b^u_j(t,x)\xi_j+c^u(t,x)|\xi|^2\right\}\sqrt{\FM}.
\end{equation}
Since $\FP$ is a projector, it holds that
\begin{equation*}
    \int_{\R^3}(1,\xi,|\xi|^2)\sqrt{\FM}\{\FI-\FP\} u d\xi=0,
\end{equation*}
i.e. $\{\FI-\FP\} u$ is orthogonal to $\CN$, which together with the
form \eqref{form.P} of $\FP$ imply
\begin{eqnarray}
  a^u &=& \frac{1}{2}\int_{\R^3}(5-|\xi|^2)\sqrt{\FM} ud\xi,\label{def.a}\\
  b^u &=& \int_{\R^3}\xi\sqrt{\FM} ud\xi,\label{def.b}\\
  c^u &=& \frac{1}{6}\int_{\R^3}(|\xi|^2-3)\sqrt{\FM} ud\xi,\label{def.c}
\end{eqnarray}
where $b^u=(b^u_1,b^u_2,b^u_3)$.  Thus, $u(t,x,\xi)$ can be uniquely
decomposed into
\begin{equation}\label{mmd}
\left\{\begin{array}{l}
 u(t,x,\xi)= \FP u\oplus \{\FI-\FP\}u,\\[3mm]
 \FP
    u=\left\{a^u(t,x)+b^u(t,x)\cdot \xi+c^u(t,x)|\xi|^2\right\}\sqrt{\FM}\in\CN,\\[3mm]
\{\FI-\FP\}u\in\CN^\perp,
\end{array}\right.
\end{equation}
where $\FP u$ is called the macroscopic component of $u(t,x,\xi)$
with coefficients $(a^u,b^u,c^u)$, and $\{\FI-\FP\}u$ the
microscopic
 component of $u(t,x,\xi)$. For later use, one can further decompose
 $\FP u$ as
\begin{equation}\label{mmd.pu}
\left\{\begin{array}{l}
 \dis \FP u= \FP_0 u\oplus \FP_1 u,\\[3mm]
 \dis \FP_0 u=(a^u+3c^u)\sqrt{\FM},\ \ a^u+3c^u=\int_{\R^3} \sqrt{\FM} ud\xi,\\[3mm]
\dis \FP_1u= \{b^u\cdot \xi+c^u(|\xi|^2-3)\}\sqrt{\FM},
\end{array}\right.
\end{equation}
where $\FP_0$ and $\FP_1$ are the projectors corresponding to the
hyperbolic and parabolic parts of the macroscopic (fluid) component,
respectively.

In what follows, we shall apply the macro-micro decomposition
\eqref{mmd} to the system \eqref{VPB.p1}-\eqref{VPB.p2} to deduce
the macroscopic balance laws satisfied by $(a^u,b^u,c^u)$. Firstly,
by multiplying \eqref{VPB.1} by the collision invariants
$1,\xi,|\xi|^2$, one can get the local conservation laws
\begin{eqnarray*}
  \pa_t\int_{\R^3} fd\xi+\na_x\cdot\int_{\R^3}\xi fd\xi &=& 0,\\
  \pa_t\int_{\R^3} \xi fd\xi+\na_x\cdot\int_{\R^3} \xi\otimes\xi
  fd\xi
  -\na_x \Phi \int_{\R^3} fd\xi&=&0,\\
  \pa_t\int_{\R^3}|\xi|^2fd\xi+\na_x\cdot\int_{\R^3}|\xi|^2\xi
  fd\xi
  -\na_x \Phi \cdot \int_{\R^3}\xi fd\xi&=&0.
\end{eqnarray*}
Plugging $f=\FM+ \sqrt{\FM}\FP u+ \sqrt{\FM}\{\FI-\FP\}u$ into the
above equations as well as into the Poisson equation \eqref{VPB.2}
gives
\begin{eqnarray}
  \pa_t (a^u+3c^u)+\na_x\cdot b^u &=& 0,\label{cl.0}\\
  \pa_t b^u+\na_x (a^u+5c^u)+\na_x\cdot \langle
  \xi\otimes\xi\sqrt{\FM},\{\FI-\FP\}u\rangle-\na_x\Phi &=&(a^u+3c^u)\na_x\Phi,\label{cl.1}\\
   \pa_t(3a^u+15c^u)+5\na_x\cdot b^u+\na_x\cdot  \langle
    |\xi|^2\xi\sqrt{\FM},\{\FI-\FP\}u\rangle &=&b^u\cdot \na_x\Phi,\label{cl.2}
\end{eqnarray}
and
\begin{equation}\label{cl.Poisson}
    \De_x\Phi=a^u+3c^u.
\end{equation}
Here and hereafter we  use the moment values of the normalized
global Maxwellian $\FM$:
\begin{eqnarray*}
&&\langle 1, \FM\rangle=1,\\
&&\langle |\xi_j|^2, \FM\rangle=1,\ \ \langle |\xi|^2, \FM\rangle=3,\\
&&\langle |\xi_j|^2|\xi_m|^2, \FM\rangle=1, \ \ j\neq m,\\
&&\langle |\xi_j|^4, \FM\rangle=3,\ \ \langle |\xi|^2|\xi_j|^2,
\FM\rangle=5,\ \ \langle |\xi|^4, \FM\rangle=15,\\
&&\langle |\xi|^4|\xi_j|^2, \FM\rangle=35, \ \ \langle |\xi|^6,
\FM\rangle=105.
\end{eqnarray*}
Notice that \eqref{cl.0} and \eqref{cl.2} implies
\begin{equation*}
    \pa_t c^u+\frac{1}{3}\na_x\cdot b^u +\frac{1}{6}\na_x \cdot \langle
    |\xi|^2\xi\sqrt{\FM},\{\FI-\FP\}u\rangle =0,
\end{equation*}
which is more convenient to be used to replace the time derivative
of $c^u$ in the proof.

As in \cite{D-09PD}, we also have to consider the evolution of
higher-order moments of $\{\FI-\FP\}u$:
\begin{equation*}
\langle
  \xi\otimes\xi\sqrt{\FM},\{\FI-\FP\}u\rangle, \ \ \langle
   |\xi|^2\xi\sqrt{\FM},\{\FI-\FP\}u\rangle,
\end{equation*}
which have appeared in \eqref{cl.1} and \eqref{cl.2}, respectively.
The equation \eqref{VPB.p1} can be rewritten as
\begin{eqnarray}
\pa_t \FP u+\xi\cdot \na_x \FP u-\na_x\Phi\cdot \xi \sqrt{\FM} =
-\pa_t \{\FI-\FP\}u+R+G,\label{eq.Pu}
\end{eqnarray}
where the linear term $R$ and the nonlinear term $G$ are denoted by
\begin{eqnarray}
R &=& -\xi\cdot \na_x \{\FI-\FP\}u+\FL \{\FI-\FP\}u,\label{l.form}\\
G&=&\Ga(u,u)-\na_x\Phi\cdot \na_\xi u+
\left(\frac{1}{2}\xi\cdot \na_x\Phi \right)
u .\label{n.form}
\end{eqnarray}
One can use \eqref{mmd} to further write \eqref{eq.Pu} as
\begin{eqnarray}
&&\pa_ta^u\sqrt{\FM} +\sum_{j}\{\pa_t b_j^u+\pa_j a^u-\pa_j
\Phi\}\xi_j\sqrt{\FM}
+\sum_{j}\{\pa_t c^u+\pa_jb_j^u\}|\xi_j|^2\sqrt{\FM}\nonumber\\
&&\hspace{2cm}+\sum_{j<m}\{\pa_jb_m^u+\pa_m b_j^u
\}\xi_j\xi_m\sqrt{\FM}
+\sum_{j}\pa_j c^u|\xi|^2\xi_j\sqrt{\FM}\nonumber\\
&&=-\pa_t \{\FI-\FP\}u+R+G.\label{eq.Pu.abc}
\end{eqnarray}
Define the high-order moment functions $A=(A_{jm})_{3\times 3}$ and
$B=(B_1,B_2,B_3)$ by
\begin{eqnarray}
  A_{jm}(u) = \langle(\xi_j\xi_m-1)\sqrt{\FM}, u\rangle,\ \
  B_j(u)=\langle(|\xi|^2-5)\xi_j\sqrt{\FM},
  u\rangle.\label{AB.moment}
\end{eqnarray}
Applying $A_{jm}(\cdot)$ and $B_j(\cdot)$ to both sides of
\eqref{eq.Pu.abc} and using the conservation law of mass
\eqref{cl.0}, one has
\begin{eqnarray}
\pa_t[A_{jj}(\{\FI-\FP\}u)+2c^u]+2\pa_jb_j^u&=&A_{jj}(R+G),\label{HM1}\\
\pa_t A_{jm}(\{\FI-\FP\}u)+\pa_jb_m^u+\pa_m b_j^u &=&
A_{jm}(R+G),\ \ j\neq m,\label{HM2}\\
\pa_t B_j (\{\FI-\FP\}u) +\pa_j c^u  &=& B_j (R+G),\label{HM3}
\end{eqnarray}
where \eqref{HM2} also holds for $j>m$ since it is symmetric for
$(j,m)$ due to the symmetry of $A_{jm}$. The main observation which
initially came from \cite{Guo-IUMJ} and later \cite{D-09PD} is that
for fixed $m$, from \eqref{HM1} and \eqref{HM2}, one can deduce
\begin{eqnarray}
 &\dis-\pa_t\left[\sum_{j}\pa_j A_{jm}(\{\FI-\FP\}u)+\frac{1}{2}
 \pa_m A_{mm}(\{\FI-\FP\}u)\right]-\De_x
 b_m-\pa_m\pa_m
 b_m\nonumber\\
 &\dis =\frac{1}{2}\sum_{j\neq m}\pa_m
 A_{jj}(R+G)-\sum_{j}\pa_j
 A_{jm}(R+G).\label{HM4.no}
\end{eqnarray}

\begin{remark}
It should be pointed out that the proof of \eqref{a.pro.pri.5} in
Lemma \ref{a.pro.pri} for the macroscopic dissipation is only based
on similar moment equations in presence of external forcing
corresponding to \eqref{cl.0}-\eqref{cl.2}, \eqref{cl.Poisson},
\eqref{HM1}-\eqref{HM3} and \eqref{HM4.no}. The derivation of these
moment equations was inspired by \cite{Guo-IUMJ} and firstly given
by \cite{Duan} and \cite{D-09PD} in the case of the Boltzmann
equation. They play a key role in the choice of the solution space
without time derivatives. Actually, the method of excluding time
derivatives in \cite{Duan,D-09PD,DY-09VPB} is much useful since it
will also be applied  in Section \ref{sec.ld} of this paper to the
study of the linearized system in order to obtain the time-decay
property of solutions.
\end{remark}

Next, as mentioned at the beginning of this section, we claim that
Theorem \ref{thm.cr} is indeed implied by the following proposition
in terms of perturbation $u$. Hence, the rest of this paper is
devoted to the proof of this proposition.

\begin{proposition}\label{prop.main}
Let $N\geq 4$ and $w(\xi)=(1+|\xi|^2)^{1/2}$. Assume  that
$f_0\equiv \FM+\sqrt{\FM}u_0\geq 0$ and $\|u_0\|_{H^N\cap L^2_w\cap
Z_1}$ is sufficiently small. Let $f\equiv \FM+\sqrt{\FM} u\geq 0$ be
the solution obtained in Proposition \ref{a.prop.exi} to the Cauchy
problem \eqref{VPB.p1}, \eqref{VPB.p2} and \eqref{VPB.pID} under the
assumption \eqref{ass.backg}. Then, $u$ enjoys the estimate with
algebraic decay rate in time:
\begin{equation}\label{prop.main.1}
 \left\|u(t)\right\|_{H^N}\leq C
 \left\|u_0\right\|_{H^N\cap Z_1}
  (1+t)^{-\frac{1}{4}}.
\end{equation}
Furthermore, under the following additional conditions on $u_0$, $u$
also enjoys some estimates with extra decay rates in time:

\medskip

\noindent{Case 1.} If
\begin{equation}\label{prop.main.ass}
   \FP_0 u_0\equiv 0
\end{equation}
holds for any $x\in\R^3$, then one has
\begin{equation}\label{prop.main.2}
 \left\|u(t)\right\|_{H^N}\leq C
 \left\|u_0\right\|_{H^N\cap Z_1}
  (1+t)^{-\frac{3}{4}}.
\end{equation}

\medskip

\noindent{Case 2.} Suppose that $\|u_0\|_{H^N_w\cap Z_1}$ is further
sufficiently small. Then, for any $0<\eps\leq 3/4$, there is
$\eta>0$ depending only on $\eps$ such that whenever
\begin{equation*}
    \left\|u_0\right\|_{H^N\cap Z_1}\leq \eta,
\end{equation*}
one has
\begin{equation}
 \dis  \left\|\{\FI-\FP_1\} u(t)\right\|_{H^N_w}+\|\na_x\FP_1 u(t)
 \|_{L^2_\xi(H^{N-1}_x)}\leq  C\left\|u_0\right\|_{ H^N_w\cap
    Z_1} (1+t)^{-\frac{3}{4}+\eps}.\label{prop.main.3}
\end{equation}

\end{proposition}

\noindent{\bf Proof of Theorem \ref{thm.cr}:}\ By the definition
\eqref{def.per.f} for perturbation $u$, \eqref{thm.cr.2} and
\eqref{thm.cr.4} are equivalent with \eqref{prop.main.1} and
\eqref{prop.main.2}, respectively, and conditions \eqref{thm.cr.1}
and \eqref{thm.cr.5} coincide with those in Proposition
\ref{prop.main}, and the condition \eqref{thm.cr.3} is equivalent
with \eqref{prop.main.ass} since by \eqref{mmd.pu},
\begin{equation*}
    \FP_0 u_0=\int_{\R^3} \sqrt{\FM} u_0 d\xi \sqrt{\FM}=\int_{\R^3}(f_0-\FM)
    d\xi \sqrt{\FM}
\end{equation*}
holds for any $(x,\xi)\in \R^3\times\R^3$. Finally, \eqref{thm.cr.7}
is implied by \eqref{prop.main.3}. This can be seen from
\begin{equation*}
    \|\{\FI-\FP_1\} u(t)\|_{H^N_w}\sim   \|\{\FI-\FP\}
    u(t)\|_{H^N_w}+\|\FP_0u(t)\|_{L^2_\xi(H^N_x)},
\end{equation*}
and further from
\begin{eqnarray*}
\FP_0 u= \langle e_0, f-\FM\rangle e_0 \sqrt{\FM}, \ \ \FP_1
u=\sum_{j=1}^4\langle e_j,f\rangle e_j\sqrt{\FM},
\end{eqnarray*}
and
\begin{eqnarray*}
\{\FI-\FP\} u &=& u- \sum_{j=0}^4 \langle  e_j, \sqrt{\FM}u\rangle
  e_j\sqrt{\FM}\\
  &=&\frac{f-\FM}{\sqrt{\FM}}- \sum_{j=0}^4 \langle  e_j, f-\FM\rangle
  e_j\sqrt{\FM}\\
  &=&\frac{f-\sum_{j=0}^4\langle e_j,f\rangle
  e_j\FM}{\sqrt{\FM}}.
\end{eqnarray*}
Here, the orthonormal set  $\{e_j\}_{j=0}^4$ is defined by
\eqref{thm.cr.8}. Therefore, Theorem \ref{thm.cr} is proved.

\medskip

\begin{remark}\label{rem.AssTwo}
Let us explain that conditions of Proposition \ref{prop.main} make
sure those of Proposition \ref{a.prop.exi} are satisfied. Actually,
this follows from the fact that
\begin{equation*}
    \CE(u_0)\leq C \|u_0\|_{H^N\cap Z_1}^2,
\end{equation*}
where $\CE(u_0)$ is defined in \eqref{a.prop.exi.1}. It suffices to
verify
\begin{equation}\label{AssTwo.1}
   \|\na_x \De_x^{-1}\FP_0 u_0\|^2\leq C \|u_0\|_{L^2\cap Z_1}^2.
\end{equation}
In fact, it follows from the definition of $\FP_0$ and  an
interpolation inequality that
\begin{eqnarray*}
  \|\na_x \De_x^{-1}\FP_0 u_0\|^2 \sim  \|\na_x\Phi_0\|^2
  \leq
  C\|\rho_0\|_{L^2_x}^{\frac{2}{3}}\|\rho_0\|_{L^1_x}^{\frac{4}{3}},
\end{eqnarray*}
where $\Phi_0$ and $\rho_0$ are defined by
\begin{equation*}
     \Phi_0=-\frac{1}{4\pi|x|}\ast \rho_0,\ \ \rho_0=\int_{\R^3} \sqrt{\FM} u_0 d\xi.
\end{equation*}
Furthermore, one has
\begin{eqnarray*}
 \|\rho_0\|_{L^2_x} \leq \|u_0\|,\ \   \|\rho_0\|_{L^1_x} \leq \|u_0\|_{Z_1}.
\end{eqnarray*}
Then, \eqref{AssTwo.1} follows.
\end{remark}

\section{The linearized system}\label{sec.ld}

In this section, let us consider the Cauchy problem of the
linearized system with a nonhomogeneous source:
\begin{equation}\label{LCP.eq}
  \left\{\begin{array}{l}
\dis \pa_t u +\xi\cdot \na_x u-\na_x\Phi\cdot \xi\sqrt{\FM}= \FL u + h,\\[3mm]
\dis \De_x \Phi=\int_{\R^n}\sqrt{\FM} u d\xi,\ \ t>0,x\in\R^n,  \xi\in\R^n,\\[3mm]
\dis u|_{t=0}=u_0,\ \ x\in\R^n,
  \end{array}\right.
\end{equation}
where  $h=h(t,x,\xi)$ and $u_0=u_0(x,\xi)$ are given, the spatial
dimension $n\geq 1$ is supposed to be arbitrary in order to see how
it enters into the time-decay rate at the level of linearization,
and for simplicity we still use $\FM$ to denote the normalized
$n$-dimensional Maxwellian
\begin{equation*}
    \FM=\frac{1}{(2\pi)^{n/2}} e^{-|\xi|^2/2}.
\end{equation*}
Formally, the solution to the Cauchy problem \eqref{LCP.eq} can be
written as the mild form
\begin{equation}\label{ls.so}
    u(t)=e^{ t\FB }u_0+\int_0^t e^{(t-s) \FB } h(s)ds,
\end{equation}
where $e^{t\FB}$ denotes the solution operator to the Cauchy problem
of the linearized equation without source corresponding to
\eqref{LCP.eq} for $h\equiv 0$. The goal of this section is to show
that $e^{t\FB }$ has the proposed algebraic decay properties as time tends to
infinity. The idea of proofs is to make energy estimates for
pointwise time $t$ and frequency variable $k$ which corresponds to
the spatial variable $x$. To the end, for $1\leq q\leq 2$ and
integer $m$, set the index $ \si_{q,m}$ of the time-decay rate  by
\begin{equation*}
    \si_{q,m}=\frac{n}{2}\left(\frac{1}{q}-\frac{1}{2}\right)+\frac{m}{2}.
\end{equation*}
As mentioned  at the end of Remark \ref{rem.cr.1}, another main
result of this paper is stated in the following

\begin{theorem}\label{thm.liDe}
Let $1\leq q\leq 2$ and $n\geq 1$. Let $\FP$ and $\FP_0$ be defined
in \eqref{ld.decom}.

\medskip

\noindent (i) For any $\al, \al'$ with $\al'\leq \al$, and for any
$u_0$ satisfying $\pa_x^{\al}u_0\in L^2$ and $\pa_x^{\al'} u_0\in
Z_q$, one has
\begin{equation}
 \|\pa_x^\al e^{ t\FB }u_0\|+\|\pa_x^\al \na_x\De_x^{-1}\FP_0  e^{ t\FB }u_0\|
 \leq C (1+t)^{-\si_{q,m-1}}(\|\pa_x^{\al'}u_0\|_{Z_q}+\|\pa_x^\al
    u_0\|),\label{thm.liDe.1}
\end{equation}
and
\begin{eqnarray}
  &&\dis \|\pa_x^\al e^{ t\FB }\{\FI-\FP_0\}u_0\|+\|\pa_x^\al \na_x\De_x^{-1}\FP_0
  e^{ t\FB }\{\FI-\FP_0\}u_0\|\nonumber\\
  &&\hspace{2cm}\leq C (1+t)^{-\si_{q,m}}(\|\pa_x^{\al'}\{\FI-\FP_0\}u_0\|_{Z_q}
  +\|\pa_x^\al
    \{\FI-\FP_0\}u_0\|),\label{thm.liDe.2}
\end{eqnarray}
for $t\geq 0$ with $m=|\al-\al'|$, where $C$ is a positive constant
depending only on $n,m,q$.

\medskip

\noindent (ii) Similarly, for any $\al, \al'$ with $\al'\leq \al$,
and for any $h$ satisfying  $\nu(\xi)^{-1/2}\pa_x^{\al}h(t)\in L^2$
and $\nu(\xi)^{-1/2}\pa_x^{\al'} h(t)\in Z_q$ for $t\geq 0$, one has
\begin{eqnarray}
&&\dis \left\|\pa_x^\al\int_0^t e^{ (t-s)\FB  }
\{\FI-\FP\}h(s)ds\right\|^2+\left\|\pa_x^\al\na_x\De_x^{-1}\FP_0
\int_0^t e^{ (t-s)\FB  }
\{\FI-\FP\}h(s)ds\right\|^2\nonumber\\
&& \leq
    C\int_0^t (1+t-s)^{-2\si_{q,m}}\nonumber\\
 &&\hspace{1cm}\dis(\|\nu^{-1/2}\pa_x^{\al'}\{\FI-\FP\}h(s)\|_{Z_q}^2+\|\nu^{-1/2}\pa_x^\al
    \{\FI-\FP\}h(s)\|^2)ds,\label{thm.liDe.3}
\end{eqnarray}
for $t\geq 0$ with $m=|\al-\al'|$, where $C$ is a positive constant
depending only on $n,m,q$.
\end{theorem}

\begin{remark}
In the case of the Boltzmann equation \cite{Ukai-1974} or the
Navier-Stokes-Poisson system \cite{LMZ-NSP}, some similar algebraic
time-decay estimates in Theorem \ref{thm.liDe} were obtained on the
basis of the spectral analysis of the solution semigroup. However,
so far there has not been known results applying the direct spectral
analysis as in \cite{Ukai-1974} to the study of the VPB system.
Here, the proof of Theorem \ref{thm.liDe} that we shall show can
provide a robust method to get the time-decay estimates in $L^2$
space for not only the VPB system but also some other kinetic models
such as the Boltzmann equation and Landau equation, relativistic or
non-relativistic, which is a future research goal in the general
framework as in \cite{Vi}.
\end{remark}

\begin{remark}\label{rem.extension}
In the case when the spatial domain is a torus $\T^n$,
Glassey-Strauss \cite{GS-TTSP} and Mouhot-Neumann \cite{MN} obtained
the exponential time-decay rate
\begin{equation*}
    \left\|e^{t\FB} \{\FI-\FP\} u_0\right\|_{X}\leq C e^{-\la
    t}\|u_0\|_X,
\end{equation*}
where $X=L^2(\T^n\times\R^n)$ in \cite{GS-TTSP} and
$X=H^1(\T^n\times\R^n)$ in \cite{MN}. See also \cite{Vi,DMS,D-09PD}
for the other kinetic models on torus. Actually, the proof of
Theorem \ref{thm.liDe} can be modified in a simple way so that the
above inequality with exponential rates is recovered. We shall
further explain this point at the end of this section; see Theorem
\ref{thm.torus}.
\end{remark}

To prove Theorem \ref{thm.liDe}, let $u(t)$, formally defined by
\eqref{ls.so}, be the solution to the linearized non-homogeneous
Cauchy problem \eqref{LCP.eq}, where $u_0(x,\xi)$ and $h(t,x,\xi)$
with $\FP h \equiv 0$ are given.  We decompose $u$ as
\begin{equation}\label{ld.decom}
\left\{\begin{array}{l}
\dis u(t,x,\xi)= \FP u\oplus\{\FI-\FP\}u,\\[3mm]
\dis \FP u=\FP_0 u\oplus \FP_1u\equiv \{a^u+b^u\cdot \xi+c^u|\xi|^2\}\sqrt{\FM},\\[3mm]
\dis \FP_0 u=(a^u+nc^u)\sqrt{\FM},\\[3mm]
\dis \FP_1 u=[b^u\cdot \xi+c^u(|\xi|^2-n)]\sqrt{\FM},
\end{array}\right.
\end{equation}
where $a^u,b^u,c^u$ are the macro moment functions of $u$ given by
\begin{eqnarray*}
  a^u &=& \frac{1}{2}\int_{\R^3}[(n+2)-|\xi|^2]\sqrt{\FM} ud\xi,\\
  b^u &=& \int_{\R^3}\xi\sqrt{\FM} ud\xi,\\
  c^u &=& \frac{1}{2n}\int_{\R^3}(|\xi|^2-n)\sqrt{\FM} ud\xi,
\end{eqnarray*}
which are the n dimensional generalizations of \eqref{def.a}, \eqref{def.b},
\eqref{def.c} when $n=3$. Notice that although $a^u,b^u$ and $c^u$
also depend on the spatial dimension, we used the same notations as
in Section  \ref{sec.MMD} for simplicity.  Then, from the same
procedure as in Section \ref{sec.MMD}, one has the macroscopic
balance laws satisfied by $a^u, b^u,c^u$:
\begin{eqnarray}
&& \pa_t (a^u+nc^u) + \na_x \cdot b^u =0,\label{Mleq.0}\\
&&\dis \pa_t b^u_j +\pa_j (a^u+nc^u)+2\pa_j c^u
+\sum_{m} \pa_m A_{jm}(\{\FI- \FP\}u)-\pa_j\Phi=0,\label{Mleq.1}\\
&&\pa_t c^u+\frac{1}{n}\na_x\cdot b^u +\frac{1}{2n}\sum_{j}\pa_j
B_j(\{\FI- \FP\}u)=0,\label{Mleq.2}\\
&&\De_x\Phi=a^u+nc^u,\label{Mleq.p}
\end{eqnarray}
and
\begin{eqnarray}
&&\pa_t[A_{jj}(\{\FI- \FP\}u)+2c^u]+2\pa_j b^u_j =
A_{jj}(R+h),\label{Mleq.aii}\\
&& \pa_tA_{jm}(\{\FI- \FP\}u)+\pa_jb^u_m +\pa_m b^u_j=A_{jm}(R+h),\
j\neq m,\label{Mleq.aij}\\
&& \pa_t B_j(\{\FI- \FP\}u)+\pa_j c^u=B_j(R+h),\label{Mleq.bi}
\end{eqnarray}
for $1\leq j,m\leq n$. Here, the velocity moment functions
$A_{jm}(\cdot)$ and $B_j(\cdot)$ are  given by
\begin{eqnarray*}
  A_{jm}(u) = \langle(\xi_j\xi_m-1)\sqrt{\FM}, u\rangle,\ \
  B_j(u)=\langle[|\xi|^2-(n+2)]\xi_j\sqrt{\FM},
  u\rangle,
\end{eqnarray*}
where these moment functions correspond to \eqref{AB.moment} for
$n=3$, and we again used the same notations as in Section
\ref{sec.MMD} for simplicity. $R$ has the same form with
\eqref{l.form}, given by
\begin{eqnarray}
  R &=& -\xi\cdot \na_x \{\FI-
  \FP\}u+\FL\{\FI-\FP\}u.\label{def.ll}
\end{eqnarray}
Notice that the source term $h$ does not appear in the first $n+2$
equations \eqref{Mleq.0}-\eqref{Mleq.1} due to $\FP h =0$.
Furthermore, similar to derive \eqref{HM4} from \eqref{HM1} and
\eqref{HM2}, it follows from \eqref{Mleq.aii} and \eqref{Mleq.aij}
that
\begin{eqnarray}
 &\dis-\pa_t\left[\sum_{j}\pa_j A_{jm}(\{\FI-\FP\}u)+\frac{1}{2}\pa_m A_{mm}
 (\{\FI-\FP\}u)\right]-\De_x b_m^u
 -\pa_m\pa_m
 b_m^u\nonumber\\
 &\dis =\frac{1}{2}\sum_{j\neq m}\pa_m
 A_{jj}(R+h)-\sum_{j}\pa_j
 A_{jm}(R+h).\label{HM4}
\end{eqnarray}

\begin{lemma}\label{lem.li.fr}
There is a free energy functional $E_{free}(\widehat{u}(t,k))$ which
is local in the time and frequency and takes the form of
\begin{eqnarray}
&& E_{free}(\widehat{u}(t,k))  \nonumber\\
&&= \kappa_1 \sum_m \left(\sum_j
  \frac{i k_j}{1+|k|^2}
A_{jm}(\{\FI-\FP\}\widehat{u})+\frac{1}{2} \frac{i k_m}{1+|k|^2}
A_{mm}(\{\FI-\FP\}\widehat{u})\mid-b^{\widehat{u}}_m\right)\nonumber\\
&&\ \ \ +\kappa_1 \sum_{j}\left( B_j (\{\FI-\FP\}\widehat{u})\mid
\frac{i k_j}{1+|k|^2} c^{\widehat{u}}\right)\nonumber\\
&&\ \ \ +\sum_m \left(b^{\widehat{u}}_m\mid\frac{ i k_m}{1+|k|^2}
(a^{\widehat{u}}+nc^{\widehat{u}})\right)\label{lem.li.fr.1}
\end{eqnarray}
for some constant $\kappa_1>0$, such that one has
\begin{eqnarray}
&\dis \pa_t \rmre\, E_{free}(\widehat{u}(t,k))+ \la
\frac{|k|^2}{1+|k|^2}\left(|b^{\widehat{u}}|^2+|c^{\widehat{u}}|^2
\right)+\la |a^{\widehat{u}}+nc^{\widehat{u}}|^2\nonumber\\
&\dis \leq C\left(\|\{\FI-\FP\}\widehat{u}\|_{L^2_\xi}^2
+\|\nu^{-1/2}\{\FI-\FP\}\widehat{h}\|_{L^2_\xi}^2\right)\label{lem.li.fr.2}
\end{eqnarray}
for any $t\geq 0$ and $k\in\R^n$.
\end{lemma}

\begin{proof}
We shall make estimates on $b^{\widehat{u}},c^{\widehat{u}}$ and
$a^{\widehat{u}}+nc^{\widehat{u}}$ individually and then take the
proper linear combination to deduce the desired free energy
inequality \eqref{lem.li.fr.2}. Firstly, notice that
\begin{equation*}
  \CF a^u=a^{\CF u},\ \  \CF b^u=b^{\CF u},\ \ \CF c^u=c^{\CF u},
\end{equation*}
and likewise for the high-order moment functions $A_{jm}(\cdot)$ and
$B_j(\cdot)$.

\medskip

\noindent{\it Estimate on $b^{\widehat{u}}$.}  We claim that for any
$0<\de_1<1$, it holds that
\begin{eqnarray}
&&\pa_t {\rmre}\sum_m \left(\sum_j i k_j
A_{jm}(\{\FI-\FP\}\widehat{u})+\frac{1}{2}i k_m
A_{mm}(\{\FI-\FP\}\widehat{u})\mid-b^{\widehat{u}}_m\right)+(1-\de_1)|k|^2|b^{\widehat{u}}|^2\nonumber \\
&&\leq \de_1
(1+|k|^2)|a^{\widehat{u}}+nc^{\widehat{u}}|^2+\de_1|k|^2|c^{\widehat{u}}|^2\nonumber \\
&&\ \ \
+\frac{C}{\de_1}(1+|k|^2)\left(\|\{\FI-\FP\}\widehat{u}\|_{L^2_\xi}^2
+\|\nu^{-1/2}\{\FI-\FP\}\widehat{h}\|_{L^2_\xi}^2\right).\label{lem.li.fr.b01}
\end{eqnarray}
In fact, the Fourier transform of \eqref{HM4} gives
\begin{eqnarray*}
 &\dis-\pa_t\left[\sum_{j}i k_j A_{jm}(\{\FI-\FP\}\widehat{u})
 +\frac{1}{2}i
 k_m A_{mm}(\{\FI-\FP\}\widehat{u})\right]+|k|^2b_m^{\widehat{u}}
 +k_m^2
 b_m^{\widehat{u}}\nonumber\\
 &\dis =\frac{1}{2}\sum_{j\neq m}i k_m A_{jj}(\widehat{R}+\widehat{h})-\sum_{j}i
 k_j
 A_{jm}(\widehat{R}+\widehat{h}).
\end{eqnarray*}
Taking further the complex inner product with $b_m^{\widehat{u}}$
gives
\begin{eqnarray}
 &&\dis\pa_t\left(\sum_{j}i k_j A_{jm}(\{\FI-\FP\}\widehat{u})
 +\frac{1}{2}i
 k_m A_{mm}(\{\FI-\FP\}\widehat{u})\mid-b_m^{\widehat{u}}\right)+(|k|^2+k_m^2)
 |b_m^{\widehat{u}}|^2\nonumber\\
 &&\dis =\left(\frac{1}{2}\sum_{j\neq m}i k_m A_{jj}(\widehat{R}+\widehat{h})-\sum_{j}i
 k_j
 A_{jm}(\widehat{R}+\widehat{h})\mid b_m^{\widehat{u}}\right)  \nonumber\\
 &&\ \ \ +\left(\sum_{j}i k_j A_{jm}(\{\FI-\FP\}\widehat{u})
 +\frac{1}{2}i
 k_m A_{mm}(\{\FI-\FP\}\widehat{u})\mid
 -\pa_tb_m^{\widehat{u}}\right)=I_1+I_2.\label{lem.li.fr.p02}
\end{eqnarray}
$I_1$ is bounded by
\begin{eqnarray*}
  I_1 &\leq & \de_1 |k|^2|b_m^{\widehat{u}}|^2+
  \frac{C}{\de_1}\sum_{jm}(|A_{jm}(\widehat{R})|^2+|A_{jm}(\widehat{h})|^2).
\end{eqnarray*}
For $I_2$, one can use the Fourier transforms of \eqref{Mleq.1} and
\eqref{Mleq.p}:
\begin{eqnarray}
&\dis \pa_t b^{\widehat{u}}_j +i k_j
[a^{\widehat{u}}+(n+2)c^{\widehat{u}}]
+\sum_{m}i k_m A_{jm}(\{\FI- \FP\}\widehat{u})-i k_j\widehat{\Phi}=0,\label{Mleq.1.F}\\
&\dis
-|k|^2\widehat{\Phi}=a^{\widehat{u}}+nc^{\widehat{u}}\label{Mleq.p.F}
\end{eqnarray}
to estimate it as
\begin{equation*}
   I_2 \leq \de_1 (1+|k|^2)
  |a^{\widehat{u}}+nc^{\widehat{u}}|^2+\de_1|k|^2
  |c^{\widehat{u}}|^2+\frac{C}{\de_1}(1+|k|^2)\sum_{jm}|A_{jm}(\{\FI-
  \FP\}\widehat{u})|^2.
\end{equation*}
On the other hand, notice from \eqref{def.ll} that
\begin{equation*}
\widehat{R}=-\xi\cdot k\{\FI-\FP\}\widehat{u}+\FL
\{\FI-\FP\}\widehat{u}
\end{equation*}
which implies
\begin{eqnarray*}
 |A_{jm}(\widehat{R})|^2&\leq &
 C(1+|k|^2)\|\{\FI-\FP\}\widehat{u}\|_{L^2_\xi}^2.
\end{eqnarray*}
Similarly it holds that
\begin{equation*}
|A_{jm}(\widehat{h})|^2\leq C\|\nu^{-1/2}
\{\FI-\FP\}\widehat{h}\|_{L^2_\xi}^2,\ \ |A_{jm}(\{\FI-
  \FP\}\widehat{u})|^2\leq C\|\{\FI-\FP\}\widehat{u}\|_{L^2_\xi}^2.
\end{equation*}
Thus, \eqref{lem.li.fr.b01} follows from taking the real part of
\eqref{lem.li.fr.p02} and plugging the estimates of $I_1,I_2$ into
it.

\medskip

\noindent{\it Estimate on $c^{\widehat{u}}$.}  We claim that for any
$0<\de_2<1$, it holds that
\begin{eqnarray}
&&\pa_t \rmre \sum_{j}\left( B_j (\{\FI-\FP\}\widehat{u})\mid i k_j
c^{\widehat{u}}\right)+(1-\de_2)|k|^2 |c^{\widehat{u}}|^2  \nonumber\\
&& \leq \de_2
|k|^2|b^{\widehat{u}}|^2+\frac{C}{\de_2}(1+|k|^2)\|\{\FI-\FP\}\widehat{u}\|_{L^2_\xi}^2
+\frac{C}{\de_2}\|\nu^{-1/2}\{\FI-\FP\}\widehat{h}\|_{L^2_\xi}^2.\label{lem.li.fr.c01}
\end{eqnarray}
In fact, similarly as before, from the Fourier transform of
\eqref{Mleq.bi}
\begin{equation*}
\pa_t B_j(\{\FI- \FP\}\widehat{u})+i k_j
c^{\widehat{u}}=B_j(\widehat{R}+\widehat{h}),
\end{equation*}
one can get
\begin{eqnarray}
&\dis \pa_t \left(B_j(\{\FI- \FP\}\widehat{u})\mid i k_j
c^{\widehat{u}}\right) +|k_j|^2
|c^{\widehat{u}}|^2  \nonumber\\
&\dis =\left(B_j(\widehat{R}+\widehat{h})\mid i k_j
c^{\widehat{u}}\right)+\left(B_j(\{\FI- \FP\}\widehat{u})\mid i k_j
\pa_t c^{\widehat{u}}\right)=I_3+I_4.\label{lem.li.fr.c02}
\end{eqnarray}
$I_3$ is bounded by
\begin{eqnarray*}
  I_3 &\leq & \de_2|k_j|^2| c^{\widehat{u}}|^2+
  \frac{C}{\de_2}\sum_j(|B_j(\widehat{R})|^2+|B_j(\widehat{h})|^2),
\end{eqnarray*}
and from the Fourier transform of \eqref{Mleq.2}
\begin{equation*}
    \pa_t c^{\widehat{u}}+\frac{1}{n}i k \cdot b^{\widehat{u}} +\frac{1}{2n}\sum_{j}i
    k_j
B_j(\{\FI- \FP\}\widehat{u})=0,
\end{equation*}
$I_4$ is bounded by
\begin{equation*}
    I_4\leq\frac{\de_2}{n}|k|^2| b^{\widehat{u}}
    |^2+\frac{C}{\de_2}|k|^2\sum_j |B_j(\{\FI- \FP\}\widehat{u})|^2.
\end{equation*}
Notice that similar to $A_{jm}$, it holds that
\begin{eqnarray*}
&\dis |B_{j}(\widehat{R})|^2\leq
 C(1+|k|^2)\|\{\FI-\FP\}\widehat{u}\|_{L^2_\xi}^2,\\
&\dis |B_{j}(\widehat{h})|^2\leq C\|\nu^{-1/2}
\{\FI-\FP\}\widehat{h}\|_{L^2_\xi}^2,\ \ |B_{j}(\{\FI-
  \FP\}\widehat{u})|^2\leq C\|\{\FI-\FP\}\widehat{u}\|_{L^2_\xi}^2.
\end{eqnarray*}
Then, \eqref{lem.li.fr.c01} follows from \eqref{lem.li.fr.c02} by
taking summation over $1\leq i\leq n$, taking the real part and then
applying the estimates of $I_3$ and $I_4$.
\medskip

\noindent{\it Estimate on $a^{\widehat{u}}+nc^{\widehat{u}}$.} We
claim that for any $0<\de_3<1$, it holds that
\begin{eqnarray}
&&\pa_t \rmre \sum_m \left(b^{\widehat{u}}_m\mid i k_m
(a^{\widehat{u}}+nc^{\widehat{u}})\right)+(1-\de_3)(1+|k|^2)|a^{\widehat{u}}+nc^{\widehat{u}}|^2  \nonumber\\
&&\leq
|k|^2|b^{\widehat{u}}|^2+\frac{C}{\de_3}|k|^2|c^{\widehat{u}}|^2
+\frac{C}{\de_3}|k|^2\|\{\FI-\FP\}\widehat{u}\|_{L^2_\xi}^2.\label{lem.li.fr.a}
\end{eqnarray}
In fact, by taking the complex inner product with $i k_j
(a^{\widehat{u}}+nc^{\widehat{u}}) $ and then taking summation over
$1\leq j\leq n$, it follows from \eqref{Mleq.1.F} that
\begin{eqnarray}
 &&\dis \pa_t \sum_j( b^{\widehat{u}}_j\mid i k_j
(a^{\widehat{u}}+nc^{\widehat{u}}) ) +|k|^2
|a^{\widehat{u}}+nc^{\widehat{u}}|^2+ \sum_j(-i k_j
\widehat{\Phi}\mid i k_j (a^{\widehat{u}}+nc^{\widehat{u}}))  \nonumber\\
&&\dis=(-2i k_j c^{\widehat{u}}\mid i k_j
(a^{\widehat{u}}+nc^{\widehat{u}}) )+ \sum_{jm}(-i
k_jA_{jm}(\{\FI-\FP\}\widehat{u})\mid i k_j
(a^{\widehat{u}}+nc^{\widehat{u}}))  \nonumber\\
&&\ \ \ + \sum_j( b^{\widehat{u}}\mid i k_j \pa_t
(a^{\widehat{u}}+nc^{\widehat{u}}) ).\label{lem.li.fr.a01}
\end{eqnarray}
Using \eqref{Mleq.p.F}, one has
\begin{eqnarray*}
 \sum_j(-i k_j
\widehat{\Phi}\mid i k_j (a^{\widehat{u}}+nc^{\widehat{u}}))&=&
\sum_j(k_j^2 \frac{a^{\widehat{u}}+nc^{\widehat{u}}}{|k|^2}\mid
a^{\widehat{u}}+nc^{\widehat{u}})=|a^{\widehat{u}}+nc^{\widehat{u}}|^2.
\end{eqnarray*}
The first two terms on the r.h.s. of \eqref{lem.li.fr.a01} are
bounded by
\begin{eqnarray*}
\de_3|k|^2|a^{\widehat{u}}+nc^{\widehat{u}}|^2+\frac{C}{\de_3}|k|^2|c^{\widehat{u}}|^2
+\frac{C}{\de_3}|k|^2\|\{\FI-\FP\}\widehat{ u}\|_{L^2_\xi}^2,
\end{eqnarray*}
while for the third term, it holds that
\begin{equation*}
\sum_j( b^{\widehat{u}}_j\mid i k_j \pa_t
(a^{\widehat{u}}+nc^{\widehat{u}}) )=\sum_j( b^{\widehat{u}}_j\mid i
k_j (-i k\cdot b^{\widehat{u}}) )=|k\cdot b^{\widehat{u}}|^2\leq
|k|^2|b^{\widehat{u}}|^2,
\end{equation*}
where we used the Fourier transform of \eqref{Mleq.0}:
\begin{equation}\label{lem.li.fr.a02}
\pa_t (a^{\widehat{u}}+nc^{\widehat{u}})+i k\cdot b^{\widehat{u}}=0.
\end{equation}
Then, putting the above estimates into \eqref{lem.li.fr.a01} and
taking the real part yields \eqref{lem.li.fr.a}.

\medskip

 Therefore, \eqref{lem.li.fr.2} follows from the proper
linear combination of \eqref{lem.li.fr.b01}, \eqref{lem.li.fr.c01}
and \eqref{lem.li.fr.a} by taking $0<\de_1,\de_2,\de_3<1$ small
enough and also $\kappa_0>0$ large enough. This completes the proof
of Lemma \ref{lem.li.fr}.\end{proof}

\begin{lemma}\label{lem.li.u}
It holds that
\begin{eqnarray}
\pa_t
\left(\|\widehat{u}\|_{L^2_\xi}+\frac{|a^{\widehat{u}}+nc^{\widehat{u}}|^2}{|k|^2}\right)+
\la \|\nu^{1/2}\{\FI-\FP\}\widehat{u}\|_{L^2_\xi}^2\leq C
\|\nu^{-1/2}\{\FI-\FP\}\widehat{h}\|_{L^2_\xi}^2\label{lem.li.u.1}
\end{eqnarray}
for any $t\geq 0$ and $k\in\R^n$.
\end{lemma}

\begin{proof}
The Fourier transform of \eqref{LCP.eq}$_1$ and \eqref{LCP.eq}$_2$
gives
\begin{equation*}
    \pa_t \widehat{u} + i \xi\cdot k \widehat{u}+ i
    \frac{k}{|k|^2}(a^{\widehat{u}}+n c^{\widehat{u}})\cdot \xi
    \sqrt{\FM}=\FL \widehat{u}+ \widehat{h}.
\end{equation*}
Further taking the complex inner product with $\widehat{u}$ and
taking the real part yield
\begin{eqnarray}
 \frac{1}{2}\pa_t \|\widehat{u}\|_{L^2_\xi}^2&+&\rmre \int_{\R^n}
    \left(i  \frac{k}{|k|^2}(a^{\widehat{u}}+n c^{\widehat{u}})\cdot \xi
    \sqrt{\FM}\mid\widehat{u}\right)d\xi  \nonumber\\
    &&\hspace{2cm}=\rmre \int_{\R^n}(\FL  \widehat{u}\mid
    \widehat{u})d\xi+ \rmre \int_{\R^n}(\widehat{h}\mid
    \widehat{u})d\xi.\label{lem.li.u.p1}
\end{eqnarray}
For the second term on the l.h.s. of \eqref{lem.li.u.p1},  from
\eqref{lem.li.fr.a02}, one has
\begin{eqnarray*}
 \rmre \int_{\R^n}
    \left(i  \frac{k}{|k|^2}(a^{\widehat{u}}+n c^{\widehat{u}})\cdot \xi
    \sqrt{\FM}\mid\widehat{u}\right)d\xi  &=& \rmre \int_{\R^n}
    \left( \frac{1}{|k|^2}(a^{\widehat{u}}+n c^{\widehat{u}})\mid-i k\cdot
    b^{\widehat{u}}\right)d\xi\\
    &=& \rmre \int_{\R^n}
    \left( \frac{1}{|k|^2}(a^{\widehat{u}}+n c^{\widehat{u}})\mid\pa_t (a^{\widehat{u}}+n
    c^{\widehat{u}})\right)d\xi\\
    &=&\frac{1}{2|k|^2}\pa_t |a^{\widehat{u}}+n c^{\widehat{u}}|^2.
\end{eqnarray*}
For two terms on the r.h.s. of \eqref{lem.li.u.p1}, one has
\begin{eqnarray*}
   \rmre \int_{\R^n}(\FL  \widehat{u}\mid
    \widehat{u})d\xi &\leq & -\la  \|\nu^{1/2}\{\FI-\FP\}\widehat{u}\|_{L^2_\xi}^2
\end{eqnarray*}
and
\begin{eqnarray*}
   \rmre \int_{\R^n}( \widehat{h}\mid
    \widehat{u})d\xi &=& \rmre \int_{\R^n}( \{\FI-\FP\}\widehat{h}\mid
    \{\FI-\FP\}\widehat{u})d\xi\\
    &\leq & \frac{\la}{2}
    \|\nu^{1/2}\{\FI-\FP\}\widehat{u}\|_{L^2_\xi}^2+C\|\nu^{-1/2}\{\FI-\FP\}\widehat{h}\|_{L^2_\xi}^2,
\end{eqnarray*}
where $\FP h \equiv 0$ was used. Plugging the above estimates into
\eqref{lem.li.u.p1} gives  \eqref{lem.li.u.1}. This completes the
proof of Lemma \ref{lem.li.u}.
\end{proof}

\medskip

\noindent{\bf {Proof of Theorem \ref{thm.liDe}}:} Let $\kappa_2>0$
be a small constant to be determined later. Define
\begin{equation*}
    E(\widehat{u}(t,k))=\|\widehat{u}(t,k)\|_{L^2_\xi}^2+
    \frac{1}{|k|^2}|a^{\widehat{u}}+nc^{\widehat{u}}|^2+\kappa_2\rmre\,
    E_{free} (\widehat{u}(t,k))
\end{equation*}
for $t\geq 0$ and $k\in\R^n$, where $E_{free} (\widehat{u}(t,k))$ is
given by \eqref{lem.li.fr.1}. Notice from \eqref{lem.li.fr.1} that
\begin{eqnarray*}
\left|E_{free} (\widehat{u}(t,k)) \right| &\leq &
C(|b^{\widehat{u}}|^2+|c^{\widehat{u}}|^2+
|a^{\widehat{u}}+nc^{\widehat{u}}|^2)\\
&&+\sum_{ij}(|A_{ij}(\{\FI-\FP\}\widehat{u})|^2
+|B_i(\{\FI-\FP\}\widehat{u})|^2)\\
&\leq & C(\|\FP
\widehat{u}\|_{L^2_\xi}^2+\|\{\FI-\FP\}\widehat{u}\|_{L^2_\xi}^2)\\
&\leq& C \|\widehat{u}(t,k)\|_{L^2_\xi}^2.
\end{eqnarray*}
Therefore, one can choose $\kappa_2>0$ small enough such that
\begin{equation}\label{thm.liDe.p0}
 E(\widehat{u}(t,k))\sim\|\widehat{u}(t,k)\|_{L^2_\xi}^2+
    \frac{1}{|k|^2}|a^{\widehat{u}}+nc^{\widehat{u}}|^2
\end{equation}
holds. By further letting $\kappa_2>0$ be small enough, the linear
combination of  \eqref{lem.li.u.1} and  \eqref{lem.li.fr.2} implies
\begin{eqnarray*}
  \pa_t E(\widehat{u}(t,k)) &+& \la \|\nu^{1/2}\{\FI-\FP\}\widehat{u}\|_{L^2_\xi}^2+ \la
\frac{|k|^2}{1+|k|^2}\left(|b^{\widehat{u}}|^2+|c^{\widehat{u}}|^2
\right)\\
&&\hspace{2cm}+\la |a^{\widehat{u}}+nc^{\widehat{u}}|^2\leq C
\|\nu^{-1/2}\{\FI-\FP\}\widehat{h}\|_{L^2_\xi}^2.
\end{eqnarray*}
Notice that
\begin{eqnarray*}
&&\|\nu^{1/2}\{\FI-\FP\}\widehat{u}\|_{L^2_\xi}^2
+\frac{|k|^2}{1+|k|^2}\left(|b^{\widehat{u}}|^2+|c^{\widehat{u}}|^2
\right)+|a^{\widehat{u}}+nc^{\widehat{u}}|^2\\
&&\geq \la
\frac{|k|^2}{1+|k|^2}\left(\|\{\FI-\FP\}\widehat{u}\|_{L^2_\xi}^2+|a^{\widehat{u}}+nc^{\widehat{u}}|^2
+|b^{\widehat{u}}|^2+|c^{\widehat{u}}|^2+\frac{1}{|k|^2}|a^{\widehat{u}}+nc^{\widehat{u}}|^2\right)\\
&&\geq\la
\frac{|k|^2}{1+|k|^2}\left(\|\{\FI-\FP\}\widehat{u}\|_{L^2_\xi}^2+\|\FP\widehat{u}\|_{L^2_\xi}^2
+\frac{1}{|k|^2}|a^{\widehat{u}}+nc^{\widehat{u}}|^2\right)\\
&&\geq\la \frac{|k|^2}{1+|k|^2}\left(\|\widehat{u}\|_{L^2_\xi}^2
+\frac{1}{|k|^2}|a^{\widehat{u}}+nc^{\widehat{u}}|^2\right)
\geq
\la
\frac{|k|^2}{1+|k|^2}
E (\widehat{u}(t,k)).
\end{eqnarray*}
Then, it follows that
\begin{equation}\label{thm.liDe.p00}
 \pa_t E (\widehat{u}(t,k))+  \frac{\la |k|^2}{1+|k|^2}E
 (\widehat{u}(t,k))\leq C
\|\nu^{-1/2}\{\FI-\FP\}\widehat{h}\|_{L^2_\xi}^2,
\end{equation}
which by the Gronwall inequality, implies
\begin{eqnarray}
E(\widehat{u}(t,k))&\leq& E(\widehat{u}(0,k))e^{-
\frac{\la |k|^2}{1+|k|^2}t}  \nonumber\\
&&+C\int_0^t e^{- \frac{\la
|k|^2}{1+|k|^2}(t-s)}\|\nu^{-1/2}\{\FI-\FP\}\widehat{h}(s,k)\|_{L^2_\xi}^2ds\label{thm.liDe.p1}
\end{eqnarray}
for any $t\geq 0$ and $k\in\R^n$.

Now, to prove \eqref{thm.liDe.1} and \eqref{thm.liDe.2}, let $h=0$
so that $u(t)=e^{ t\FB }u_0$  is the solution to the Cauchy problem
\eqref{LCP.eq} and hence satisfies the estimate \eqref{thm.liDe.p1}
with $h=0$. Write $k^{\al}=k_1^{\al_1}k_2^{\al_2}\cdots
k_n^{\al_n}$. By noticing
\begin{eqnarray}
&&\|\pa_x^\al e^{ t\FB }u_0\|^2+\|\pa_x^\al \na_x\De_x^{-1}\FP_0
e^{ t\FB }u_0\|^2  \nonumber\\
&&=\int_{\R^n_k}|k^{2\al}|\cdot \|\widehat{u}(t,k)\|_{L^2_\xi}^2dk +
\int_{\R^n_k}|k^{2\al}|\cdot
 \frac{1}{|k|^2}|a^{\widehat{u}}+nc^{\widehat{u}}|^2dk  \nonumber\\
 &&\leq C \int_{\R^n_k}|k^{2\al}| \left| E(\widehat{u}(t,k))\right| dk,\label{thm.liDe.p1.1}
\end{eqnarray}
then, from \eqref{thm.liDe.p1} with $h=0$  and \eqref{thm.liDe.p0},
one has
\begin{eqnarray}
&&\|\pa_x^\al e^{ t\FB }u_0\|^2+\|\pa_x^\al \na_x\De_x^{-1}\FP_0
e^{ t\FB }u_0\|^2  \nonumber\\
&&\leq C\int_{\R^n_k}|k^{2\al}|e^{-\frac{\la|k|^2}{1+|k|^2}
t}\|\widehat{u_0}(k)\|_{L^2_\xi}^2dk+
C\int_{\R^n_k}\frac{|k^{2\al}|}{|k|^2}e^{-\frac{\la|k|^2}{1+|k|^2}
t}\|\widehat{\FP_0u_0}(k)\|_{L^2_\xi}^2dk.\label{thm.liDe.p2}
\end{eqnarray}
As in \cite{Ka}, one can further estimate the first term on the
r.h.s. of  \eqref{thm.liDe.p2} by
\begin{eqnarray}
&&\int_{\R^n_k}|k^{2\al}|e^{-\frac{\la|k|^2}{1+|k|^2}
t}\|\widehat{u_0}(k)\|_{L^2_\xi}^2dk\nonumber\\
&&\leq  \int_{|k|\leq 1}
|k^{2(\al-\al')}|e^{-\frac{\la|k|^2}{1+|k|^2}
t}|k^{2\al'}|\cdot\|\widehat{u_0}(k)\|_{L^2_\xi}^2dk+ \int_{|k|\geq
1}
e^{-\frac{\la}{2}t}|k^{2\al}|\cdot\|\widehat{u_0}(k)\|_{L^2_\xi}^2dk\nonumber\\
&&\leq C
 (1+t)^{-\frac{n}{q}+\frac{n-2|\al-\al'|}{2}}\|\pa_x^{\al'}u_0\|_{Z_q}^2+ C
e^{-\frac{\la}{2}t}\|\pa_x^\al u_0\|^2,\label{thm.liDe.p3}
\end{eqnarray}
where the H\"{o}lder and Hausdorff-Young inequalities were used, and
similarly for the second term on the r.h.s. of  \eqref{thm.liDe.p2},
it holds that
\begin{eqnarray}
&&\int_{\R^n_k}\frac{|k^{2\al}|}{|k|^2}e^{-\frac{\la|k|^2}{1+|k|^2}
t}\|\widehat{\FP_0u_0}(k)\|_{L^2_\xi}^2dk \leq
\int_{\R^n_k}\frac{|k^{2\al}|}{|k|^2}e^{-\frac{\la|k|^2}{1+|k|^2}
t}\|\widehat{u_0}(k)\|_{L^2_\xi}^2dk \nonumber\\
&&\leq C
 (1+t)^{-\frac{n}{q}+\frac{n-2(|\al-\al'|-1)}{2}}\|\pa_x^{\al'}u_0\|_{Z_q}^2+ C
e^{-\frac{\la}{2}t}\|\pa_x^\al u_0\|^2.\label{thm.liDe.p4}
\end{eqnarray}
Thus, \eqref{thm.liDe.1} follows from \eqref{thm.liDe.p2} together
with \eqref{thm.liDe.p3} and \eqref{thm.liDe.p4}. Moreover, to prove
\eqref{thm.liDe.2}, notice $\FP_0\{\FI-\FP_0\}u_0=0$ and hence
\eqref{thm.liDe.2} similarly follows only from \eqref{thm.liDe.p2}
and \eqref{thm.liDe.p3} since the second term on the r.h.s. of
\eqref{thm.liDe.p2} vanishes.

Finally, to prove \eqref{thm.liDe.3}, let $u_0=0$ so that
\begin{equation*}
    u(t)=\int_0^t e^{ (t-s)\FB  } \{\FI-\FP\}h(s)ds
\end{equation*}
 is the solution to the Cauchy problem
\eqref{LCP.eq} and hence satisfies the estimate \eqref{thm.liDe.p1}
with $u_0=0$. Then, similar to \eqref{thm.liDe.p1.1}, one has
\begin{eqnarray*}
&& \left\|\pa_x^\al\int_0^t e^{ (t-s)\FB  }
\{\FI-\FP\}h(s)ds\right\|^2+\left\|\pa_x^\al\na_x\De_x^{-1}\FP_0
\int_0^t e^{ (t-s)\FB  } \{\FI-\FP\}h(s)ds\right\|^2\\
&&\leq C\int_0^t\int_{\R^n_k}|k^{2\al}|e^{-\frac{\la|k|^2}{1+|k|^2}
(t-s)}\|\nu^{-1/2}\{\FI-\FP\}\widehat{h}(s,k)\|_{L^2_\xi}^2dkds.
\end{eqnarray*}
Therefore, \eqref{thm.liDe.3} follows in the same way as in
\eqref{thm.liDe.p3}. This completes the proof of Theorem
\ref{thm.liDe}.

\medskip

We conclude this section by extending Theorem \ref{thm.liDe} to the
exponential time-decay rate in the case of $\T^n$ as mentioned in
Remark \ref{rem.extension}.

\begin{theorem}\label{thm.torus}
Suppose
\begin{equation}\label{thm.torus.ass}
    \int_{\T^n}\FP u_0 dx=0
\end{equation}
for any $\xi\in\R^n$. Let $e^{t\FB}u_0$ be the solution to the
Cauchy problem
\begin{equation}\label{LCP.eq.T}
  \left\{\begin{array}{l}
\dis \pa_t u +\xi\cdot \na_x u-\na_x\Phi\cdot \xi\sqrt{\FM}= \FL u,\\[3mm]
\dis \De_x \Phi=\int_{\T^n}\sqrt{\FM} u d\xi,\ \int_{\T^n} \Phi dx=0,\ \ t>0,x\in\T^n,  \xi\in\R^n,\\[3mm]
\dis u|_{t=0}=u_0,\ \ x\in\T^n.
  \end{array}\right.
\end{equation}
Then, there are constants $C>0$, $\la>0$ such that
\begin{equation}\label{thm.torus.1}
    \|e^{t\FB} u_0\|_{L^2(\T^n\times\R^n)}\leq C e^{-\la
    t}\|u_0\|_{L^2(\T^n\times\R^n)},
\end{equation}
for any $t\geq 0$.
\end{theorem}

\begin{proof}
In fact, on the one hand, \eqref{lem.li.fr.2} and \eqref{lem.li.u.1}
with $h=0$ in Lemma \ref{lem.li.fr} and Lemma \ref{lem.li.u} still
hold for any $t\geq 0$ and $k\in \Z^n$ for the solution
$u(t)=e^{t\FB} u_0$ to the Cauchy problem \eqref{LCP.eq.T} in the
torus case. On the other hand, from \eqref{Mleq.0}-\eqref{Mleq.2},
one has the total conservation laws
\begin{equation*}
    \left.\int_{\T^n} (a^u,b^u,c^u) dx\right|_{t>0}= \int_{\T^n} (a^{u_0},b^{u_0},c^{u_0})
    dx,
\end{equation*}
which together with the assumption \eqref{thm.torus.ass} imply
\begin{equation*}
    \left.\int_{\T^n} (a^u,b^u,c^u) dx\right|_{t\geq 0}\equiv 0.
\end{equation*}
Thus, it follows that
\begin{equation*}
    \left.(a^{\widehat{u}},b^{\widehat{u}},c^{\widehat{u}}) \right|_{t\geq 0,k=0}\equiv
    0,
\end{equation*}
which yields
\begin{equation*}
    \frac{2|k|^2}{1+|k|^2}(|b^{\widehat{u}}|^2+|c^{\widehat{u}}|^2+
|a^{\widehat{u}}+nc^{\widehat{u}}|^2)\geq
|b^{\widehat{u}}|^2+|c^{\widehat{u}}|^2+
|a^{\widehat{u}}+nc^{\widehat{u}}|^2,
\end{equation*}
for any $t\geq 0$ and $k\in \Z^n$. Therefore, similar to get
\eqref{thm.liDe.p00}, it holds that
\begin{equation*}
 \pa_t E (\widehat{u}(t,k))+ \la E
 (\widehat{u}(t,k))\leq 0,
\end{equation*}
where $E (\widehat{u}(t,k))$ is still given by \eqref{thm.liDe.p0}.
Then,
\begin{equation}\label{thm.torus.p1}
    E(\widehat{u}(t,k))\leq e^{-\la t}  E(\widehat{u}(0,k)).
\end{equation}
holds for any $t\geq 0$ and $k\in \Z^n$. Notice that in the torus
case, one has
\begin{eqnarray*}
\int_{\Z^n}E(\widehat{u}(t,k))dk &\sim&
\int_{\Z^n}\|\widehat{u}(t,k)\|_{L^2_\xi}^2dk+
\int_{\Z^n}\frac{1}{|k|^2}|a^{\widehat{u}}+nc^{\widehat{u}}|^2dk\\
&\sim& \int_{\Z^n}\|\widehat{u}(t,k)\|_{L^2_\xi}^2dk +
\int_{\Z^n}|a^{\widehat{u}}+nc^{\widehat{u}}|^2dk\\
&\sim& \int_{\Z^n}\|\widehat{u}(t,k)\|_{L^2_\xi}^2dk
=\|u(t)\|_{L^2(\T^n\times\R^n)}^2.
\end{eqnarray*}
Therefore, the integration of \eqref{thm.torus.p1} over $k\in\Z^n$
gives \eqref{thm.torus.1}. This completes the proof of Theorem
\ref{thm.torus}.
\end{proof}

\section{The nonlinear system}\label{sec.ES}

\subsection{Energy estimates}\label{sec.ES.sub1}

From now on, we devote ourselves to the proof of Proposition
\ref{prop.main}. For that, assume that all conditions of Proposition
\ref{prop.main} hold, particularly
$$
\|u_0\|_{H^N\cap L^2_w\cap Z_1}
$$
is supposed  to be sufficiently small throughout this section. Let
$u$ be the solution to the Cauchy problem on the nonlinear VPB
system \eqref{VPB.p1}, \eqref{VPB.p2} and \eqref{VPB.pID} obtained
by Proposition \ref{a.prop.exi}. Here, notice from Remark
\ref{rem.AssTwo} that the solution $u$ indeed exists under the
assumptions of Proposition \ref{prop.main}. Therefore, from
\eqref{a.prop.exi.3},
\begin{equation}\label{smallness}
    \CE(u(t))+\la \int_0^t \CD(u(s))ds\leq \CE(u_0)
\end{equation}
holds for any $t\geq 0$, where $\CE(u(t))$ and $\CD(u(t))$ are
defined by \eqref{a.prop.exi.1} and \eqref{a.prop.exi.2},
respectively. Thus, one can suppose that the energy functional
$\CE(u(t))$
 is small enough uniformly in time. We also
remark that those uniform {\it a priori} estimates given in Lemma
\ref{a.pro.pri} will be used  later on in the proof.

In this subsection, for some preparations,  we are concerned with
energy estimates on the microscopic part $\{\FI-\FP\}u$ to obtain
some Lyapunov-type inequalities. Recall from \cite{DY-09VPB} that
$\{\FI-\FP\}u$ satisfies
\begin{eqnarray}
&& \pa_t \{\FI-\FP\}u+\xi \cdot\na_x \{\FI-\FP\}u+\na_x
\Phi\cdot\na_\xi
\{\FI-\FP\}u\nonumber\\
&=& \FL \{\FI-\FP\}u+\Ga(u,u)+\frac{1}{2}\xi\cdot\na_x\Phi \{\FI-\FP\}u\nonumber\\
&& -\{\FI-\FP\}(\xi\cdot\na_x\FP u +\na_x\Phi\cdot\na_\xi
\FP u-\frac{1}{2}\xi\cdot\na_x\Phi \FP u)\nonumber\\
&&-\FP(\xi\cdot\na_x \{\FI-\FP\} u +\na_x\Phi\cdot\na_\xi
\{\FI-\FP\} u-\frac{1}{2}\xi\cdot\na_x\Phi  \{\FI-\FP\}
u),\label{lem-micro-SVMD-p02}
\end{eqnarray}
and also the following lemma was proved in \cite{UY-AA,DUYZ}.
\begin{lemma}\label{ineq.ga.K}
It holds that
\begin{eqnarray*}
\|\nu^{-\ga} \Ga(u,v)\|_{L^2_\xi}\leq C (\|\nu^{1-\ga} u\|_{L^2_\xi}
\|v\|_{L^2_\xi}+\|u\|_{L^2_\xi}\|\nu^{1-\ga} v\|_{L^2_\xi}),
\end{eqnarray*}
for $0\leq \ga\leq 1$, and
\begin{eqnarray*}
&\dis \|K u\|_{H^N}\leq C\|u\|_{H^N},\\
&\dis \|\Ga(u,v)\|_{H^N}\leq
C(\|u\|_{H^N_\nu}\|v\|_{H^N}+\|u\|_{H^N}\|v\|_{H^N_\nu}).
\end{eqnarray*}
\end{lemma}

Firstly, corresponding to \eqref{a.pro.pri.1}, it is straightforward
to prove

\begin{lemma}\label{lem.ES}
It holds that
\begin{eqnarray}
  \frac{d}{dt}\|\{\FI-\FP\} u\|^2+\la \|w^{1/2} \{\FI-\FP\} u\|^2 &\leq
  & C \|\na_x\FP u\|^2+
  C\CE(u(t))\CD(u(t)),\label{lem.ES.0}
\end{eqnarray}
for any $t\geq 0$.
\end{lemma}

\begin{proof}
The direct zero-order energy estimate on \eqref{lem-micro-SVMD-p02}
gives
\begin{eqnarray*}
&&\frac{1}{2}\frac{d}{dt}\|\{\FI-\FP\} u\|^2+\la \|\nu^{1/2}
\{\FI-\FP\} u\|^2\\
&&\leq \int_{\R^3} \langle \Ga(u,u),\{\FI-\FP\} u\rangle
dx+\int_{\R^3} \langle \frac{1}{2}\xi\cdot \na_x \Phi, (\{\FI-\FP\}
u)^2\rangle dx\\
&&\ \ \ -\int_{\R^3} \langle \{\FI-\FP\}(\xi\cdot\na_x\FP u
+\na_x\Phi\cdot\na_\xi \FP u-\frac{1}{2}\xi\cdot\na_x\Phi \FP u),
\{\FI-\FP\} u\rangle dx,
\end{eqnarray*}
which further implies
\begin{eqnarray*}
&&\dis \frac{1}{2}\frac{d}{dt}\|\{\FI-\FP\} u\|^2+\la \|\nu^{1/2}
\{\FI-\FP\} u\|^2\\
&&\leq C\sqrt{ \CE(u(t))\CD(u(t))} \| \nu^{1/2} \{\FI-\FP\}
u\| +C \sqrt{\CE(u(t))} \| \nu^{1/2} \{\FI-\FP\} u\|^2\\
&&\ \ \ + C\|\na_x\FP u\|  \| \nu^{1/2} \{\FI-\FP\} u\|,
\end{eqnarray*}
where Lemma \ref{ineq.ga.K} and Sobolev inequality were used. Then,
\eqref{lem.ES.0} follows from the Cauchy's inequality, smallness of
$\CE(u(t))$ and the equivalence $w(\xi)\sim\nu(\xi)$. This completes
the proof of Lemma \ref{lem.ES}.
\end{proof}

Next, we consider the weighted energy estimates on $\{\FI-\FP\}u$,
whose aim is to obtain the weighted high-order Lyapunov inequalities
later.

\begin{lemma}\label{lem.ES.w}
It holds that
\begin{eqnarray}
&\dis   \frac{d}{dt}\|w^{1/2}\{\FI-\FP\} u\|^2+\la \|w \{\FI-\FP\}
u\|^2
  \leq C [1+\CE(u(t))] \CD(u(t)). \label{lem.ES.w.1}
 \end{eqnarray}
 Furthermore
\begin{eqnarray}
  &\dis \frac{d}{dt}\sum_{
1\leq |\al|\leq N }\|w^{1/2}\pa_x^\al u\|^2
  +\la \sum_{1\leq |\al|\leq N }\|w \pa_x^\al u\|^2  \nonumber\\
  &\dis \leq
C\CD(u(t)) + C \CE(u(t))\sum_{|\al|+|\be|\leq N}\| w
\pa_x^{\al}\pa_\xi^{\be}\{\FI-\FP\} u\|^2,\label{lem.ES.w.2}
\end{eqnarray}
and
\begin{eqnarray}
  &\dis \frac{d}{dt}\sum_{\substack{|\be|\geq 1 \\
|\al|+|\be|\leq N }}\|w^{1/2}\pa_x^\al\pa_\xi^\be \{\FI-\FP\} u\|^2
  +\la \sum_{\substack{|\be|\geq 1 \\
|\al|+|\be|\leq N }}\|w \pa_x^\al\pa_\xi^\be \{\FI-\FP\} u\|^2  \nonumber\\
  &\dis \leq
C\CD(u(t)) + C \CE(u(t))\sum_{|\al|+|\be|\leq N}\| w
\pa_x^{\al}\pa_\xi^{\be}\{\FI-\FP\} u\|^2,\label{lem.ES.w.3}
\end{eqnarray}
for any $t\geq 0$.

\end{lemma}

\begin{proof}
The proof of \eqref{lem.ES.w.1} is similar to that of
\eqref{lem.ES.0} and hence is omitted. To prove \eqref{lem.ES.w.2},
let $1\leq |\al|\leq N$ and then it follows from \eqref{VPB.p1} that
\begin{eqnarray}
&\dis \pa_t (\pa_x^\al u)+\xi\cdot \na_x (\pa_x^\al
u)+\na_x\Phi\cdot
\na_\xi (\pa_x^\al u)+\nu \pa_x^\al u\nonumber\\
&\dis =K\pa_x^\al u+\pa_x^\al \na_x\Phi\cdot\xi\sqrt{\FM}+\pa_x^\al
\Ga(u,u)-[\pa_x^\al,\na_x\Phi\cdot \na_\xi]u,\label{lem.ES.w.p1}
\end{eqnarray}
where the last term on the r.h.s. denotes the commutator. Then,
\eqref{lem.ES.w.2} follows from multiplying the above equation by
$w\pa_x^\al u$, taking integration over $\R^3\times\R^3$ and then
using integration by parts, Cauchy's inequality, Lemma
\ref{ineq.ga.K} and Sobolev inequality.

Finally, we prove \eqref{lem.ES.w.3}. Fix  $\al,\be$ with
$|\al|+|\be|\leq N$ and $|\be|\geq 1$. For simplicity, write
$z=\pa_x^\al\pa_\xi^\be\{\FI-\FP\} u$. Then, after applying
$\pa_x^\al\pa_\xi^\be$ to \eqref{lem-micro-SVMD-p02}, $z$ satisfies
\begin{eqnarray}
  \pa_t z+ \xi\cdot \na_x z+\na_x\Phi \cdot \na_\xi z  +\nu z&=&I, \label{lem.ES.w.p2}
\end{eqnarray}
where $I$ is denoted by
\begin{equation*}
    I=I_{1}+I_{2}+I_{3}
\end{equation*}
with
\begin{eqnarray*}
I_{1} &=& \pa_\xi^\be K\pa_x^\al \{\FI-\FP\}u+\pa_x^\al\pa_\xi^\be\Ga(u,u)
+\frac{1}{2} \pa_x^\al\pa_\xi^\be (\xi\cdot\na_x\Phi \{\FI-\FP\}u),\\
  I_{2} &=&  -\pa_x^\al\pa_\xi^\be\{\FI-\FP\}(\xi\cdot\na_x\FP u +\na_x\Phi\cdot\na_\xi
\FP u-\frac{1}{2}\xi\cdot\na_x\Phi \FP u)\nonumber\\
&&-\pa_x^\al\pa_\xi^\be\FP(\xi\cdot\na_x \{\FI-\FP\} u
+\na_x\Phi\cdot\na_\xi \{\FI-\FP\} u-\frac{1}{2}\xi\cdot\na_x\Phi
\{\FI-\FP\} u),\\
I_{3} &=& -[\pa_\xi^\be,\xi\cdot\na_x]\pa_x^\al \{\FI-\FP\}
u-[\pa_\xi^\be,\nu(\xi)]\pa_x^\al \{\FI-\FP\}
u\nonumber\\
&&-[\pa_x^\al,\na_x\Phi\cdot \na_\xi]\pa_\xi^\be \{\FI-\FP\} u.
\end{eqnarray*}
Here, notice that $I_2$ contains the macroscopic projection $\FP$
which can absorb both the weight and derivative of velocity
variable, and $I_3$ contains all the commutators in which the total
order of differentiation of $\{\FI-\FP\} u$ is no more than $N$.
Then, similarly for \eqref{lem.ES.w.p1}, the direct energy estimate
of \eqref{lem.ES.w.p2} yields \eqref{lem.ES.w.3}. This also
completes the proof of Lemma \ref{lem.ES.w}.

\end{proof}

\subsection{Time decay of energy}
In this subsection, we shall prove \eqref{prop.main.1} and
\eqref{prop.main.2} in Proposition \ref{prop.main}. For this, define
two temporal functions by
\begin{equation}\label{TE.def.e01}
    \CE_\infty^{(m)}(t)=\sup_{0\leq s\leq t} (1+s)^{\frac{1}{2}+m}
    \CE(u(s)), \ \ m=0, 1,
\end{equation}
where $\CE(u(t))$ is defined by \eqref{a.prop.exi.1}. For
simplicity, also define two constants depending only on initial data
by
\begin{equation}\label{TE.def.eps}
    \eps_0=\|u_0\|_{H^N}^2+\|u_0\|_{Z_1}^2,\ \
    \eps_{0,\nu}=\|u_0\|_{H^N}^2+\|u_0\|_{Z_1}^2+\|\nu^{1/2}u_0\|^2.
\end{equation}

\medskip

\noindent{\bf {Proof of \eqref{prop.main.1} and \eqref{prop.main.2}
in Proposition \ref{prop.main}}:}\ We begin with the proof of
\eqref{prop.main.1}. Firstly recall that from \eqref{a.prop.exi.3},
one  has the energy inequality
\begin{equation}\label{TE.p00}
    \frac{d}{dt}\CE(u(t)) +\la \CD(u(t))\leq 0,
\end{equation}
where $\CD(u(t))$ is defined by \eqref{a.prop.exi.2}. By comparing
\eqref{a.prop.exi.1} with \eqref{a.prop.exi.2}, it holds that
\begin{equation*}
    \CD(u) +\|\FP_1 u\|^2+\|\na_x\De_x^{-1}\FP_0 u\|^2\geq \la
    \CE(u).
\end{equation*}
Then, it follows from \eqref{TE.p00} that
\begin{equation}\label{TE.p01}
    \frac{d}{dt}\CE(u(t)) +\la \CE(u(t))\leq  C\|\FP_1 u(t)\|^2+C \|\na_x\De_x^{-1}\FP_0
    u(t)\|^2.
\end{equation}
Also recall that the solution $u$ to the Cauchy problem
\eqref{VPB.p1}-\eqref{VPB.pID} of the nonlinear VPB system can be
written as the mild form
\begin{equation}\label{TE.p02}
    u(t)=e^{\FB t} u_0+ \int_0^t e^{\FB (t-s)} G(s) ds,
\end{equation}
where the source term $G$ given by \eqref{n.form} is rewritten as
\begin{equation*}
    G=G_1+G_2
\end{equation*}
with
\begin{eqnarray}
  G_1 = \Ga(u,u),\ \
  G_2 = -\na_x\Phi \cdot \na_\xi u +\frac{1}{2}\xi \cdot \na_x
  \Phi u. \label{TE.def.G12}
\end{eqnarray}
By checking $\FP G_1\equiv 0$ and $\FP_0 G_2 \equiv 0$, one can
decompose $G$ as
\begin{equation}\label{TE.p03}
    G=\{\FI-\FP\} G_1+ \{\FI-\FP\} G_2 + \FP_1 G_2.
\end{equation}
By applying Theorem \ref{thm.liDe} with the spatial dimension $n=3$
to \eqref{TE.p02} and using the decomposition \eqref{TE.p03}, one
has
\begin{eqnarray}
&& \|\FP_1 u(t)\|^2+ \|\na_x\De_x^{-1}\FP_0
    u(t)\|^2 \nonumber\\
&& \leq C(1+t)^{-\frac{1}{2} } (\|u_0\|_{Z_1}^2+\|u_0\|^2) \nonumber\\
&&\ \ \ +\sum_{j=1}^2 C\int_0^t
(1+t-s)^{-\frac{3}{2}}(\|\nu^{-1/2}\{\FI-\FP\} G_j(s)\|_{Z_1}^2+
\|\nu^{-1/2}\{\FI-\FP\} G_j(s)\|^2) ds \nonumber\\
&&\ \ \ +C\left[\int_0^t (1+t-s)^{-\frac{3}{4}}(\|\FP_1 G_2
(s)\|_{Z_1}+\|\FP_1 G_2(s)\|)ds\right]^2,\label{TE.p04}
\end{eqnarray}
where the first term on the r.h.s. follows from \eqref{thm.liDe.1},
the second term from \eqref{thm.liDe.3} and the third term from
\eqref{thm.liDe.2}. One can further compute those terms of $G_1,G_2$
in \eqref{TE.p04} by
\begin{eqnarray*}
&&\|\nu^{-1/2}\{\FI-\FP\} G_1\|_{Z_1}^2+ \|\nu^{-1/2}\{\FI-\FP\}
G_1\|^2\\
&&=\|\nu^{-1/2}\Ga(u,u)\|_{Z_1}^2+ \|\nu^{-1/2}\Ga(u,u)\|^2\\
&&\leq C \|\nu^{1/2} u\|^2\|u\|^2+  C  \|\nu^{1/2} u\|^2\sup_x
\|u\|_{L^2_\xi}^2\\
&&\leq C \|\nu^{1/2} u_0\|^2 \CE(u)
\end{eqnarray*}
from Lemma \ref{ineq.ga.K}, and
\begin{eqnarray*}
&&\|\nu^{-1/2}\{\FI-\FP\} G_2\|_{Z_1}^2+ \|\nu^{-1/2}\{\FI-\FP\}
G_2\|^2\\
&&\leq C \|\na_x\Phi\|_{L^2_x\cap L^\infty_x}^2(\|\na_\xi u\|^2
+\|\nu^{1/2} u\|^2)\\
&&\leq C (\CE(u_0)+\|\nu^{1/2} u_0\|^2) \CE(u),
\end{eqnarray*}
and
\begin{equation*}
\|\FP_1 G_2 \|_{Z_1}+\|\FP_1 G_2\|\leq C  \|\na_x\Phi\|_{L^2_x\cap
L^\infty_x} \|u\|\leq C \CE(u).
\end{equation*}
Here and hereafter, one can use the uniform bounds of $\CE(u(t))$
and $\|\nu^{1/2} u(t)\|$ due to
\begin{eqnarray*}
&\dis \sup_{t\geq 0} \CE(u(t))\leq \CE(u_0)\leq  C\eps_0,\\
&\dis \sup_{t\geq 0} \|\nu^{1/2} u(t)\|^2\leq C   \|\nu^{1/2}
u_0\|^2+ \CE(u_0)\leq C\eps_{0,\nu},
\end{eqnarray*}
where the time integration of \eqref{TE.p00} and \eqref{lem.ES.w.1}
was used. Plugging the above inequalities into \eqref{TE.p04}, it
follows that
\begin{eqnarray*}
 &&\|\FP_1 u(t)\|^2+ \|\na_x\De_x^{-1}\FP_0
    u(t)\|^2\\
&&\leq C (1+t)^{-\frac{1}{2}} \eps_0  + C\eps_{0,\nu} \int_0^t
(1+t-s)^{-\frac{3}{2}}
\CE(u(s))ds \nonumber \\
&&\ \ \ +C \int_0^t \left[(1+t-s)^{-\frac{3}{4}} \CE
(u(s))ds\right]^2.
\end{eqnarray*}
By the definition of $\CE_\infty^{(0)}(t)$ in \eqref{TE.def.e01}, it
further follows that
\begin{equation}\label{TE.p05}
\|\FP_1 u(t)\|^2+ \|\na_x\De_x^{-1}\FP_0
    u(t)\|^2
\leq  C (1+t)^{-\frac{1}{2}} \left\{\eps_0  +
\eps_{0,\nu}\CE_\infty^{(0)}(t)+[\CE_\infty^{(0)}(t)]^2\right\},
\end{equation}
where we used
\begin{eqnarray*}
&&\int_0^t(1+t-s)^{-\frac{3}{2}} (1+s)^{-\frac{1}{2}}ds\leq
    C(1+t)^{-\frac{1}{2}},\\
&&\int_0^t(1+t-s)^{-\frac{3}{4}} (1+s)^{-\frac{1}{2}}ds\leq
    C(1+t)^{-\frac{1}{4}}.
\end{eqnarray*}
Due to the Gronwall inequality, \eqref{TE.p01} together with
\eqref{TE.p05} yields
\begin{eqnarray*}
  \CE(u(t)) &\leq & \CE(u_0) e^{-\la t} + C\int_0^t e^{-\la
  (t-s)}\{\|\FP_1 u(s)\|^2+ \|\na_x\De_x^{-1}\FP_0
    u(s)\|^2\}ds\\
    &\leq & C (1+t)^{-\frac{1}{2}} \left\{\eps_{0}  +
\eps_{0,\nu}\CE_\infty^{(0)}(t)+[\CE_\infty^{(0)}(t)]^2\right\},
\end{eqnarray*}
which implies
\begin{equation*}
\CE_\infty^{(0)}(t)\leq C \left\{\eps_{0}  +
\eps_{0,\nu}\CE_\infty^{(0)}(t)+[\CE_\infty^{(0)}(t)]^2\right\},
\end{equation*}
for any $t\geq 0$. Therefore, as long as both $\eps_0$ and
$\eps_{0,\nu}$ are small enough, one has
\begin{equation}\label{TE.p06}
\sup_{t\geq 0}\CE_\infty^{(0)}(t)\leq C\eps_0,
\end{equation}
which proves  \eqref{prop.main.1} by the definition of
$\CE_\infty^{(0)}(t)$ in \eqref{TE.def.e01}.

Next, one can modify the above proof of \eqref{prop.main.1} to
obtain \eqref{prop.main.2} under the assumption
\eqref{prop.main.ass}. In fact, under the additional condition
\eqref{prop.main.ass}, \eqref{TE.p04} can be refined as
\begin{eqnarray}
&& \|\FP_1 u(t)\|^2+ \|\na_x\De_x^{-1}\FP_0
    u(t)\|^2 \nonumber\\
&& \leq C(1+t)^{-\frac{3}{2} } (\|u_0\|_{Z_1}^2+\|u_0\|^2) \nonumber\\
&&\ \ \ +\sum_{j=1}^2 C\int_0^t
(1+t-s)^{-\frac{3}{2}}(\|\nu^{-1/2}\{\FI-\FP\} G_j(s)\|_{Z_1}^2+
\|\nu^{-1/2}\{\FI-\FP\} G_j(s)\|^2) ds \nonumber\\
&&\ \ \ +C\left[\int_0^t (1+t-s)^{-\frac{3}{4}}(\|\FP_1 G_2
(s)\|_{Z_1}+\|\FP_1 G_2(s)\|)ds\right]^2,\label{TE.p06.1}
\end{eqnarray}
where instead of using \eqref{thm.liDe.1} to estimate the first term
on the r.h.s. of \eqref{TE.p04}, we used \eqref{thm.liDe.2} since
$u_0$ does not contain the hyperbolic component, i.e.,
$u_0=\{\FI-\FP\} u_0$. Corresponding to obtain \eqref{TE.p05}, from
the definition of $\CE_\infty^{(1)}(t)$ in \eqref{TE.def.e01},
\eqref{TE.p06.1} implies
\begin{equation}\label{TE.p07}
\|\FP_1 u(t)\|^2+ \|\na_x\De_x^{-1}\FP_0
    u(t)\|^2
\leq  C (1+t)^{-\frac{3}{2}} \left\{\eps_0  +
\eps_{0,\nu}\CE_\infty^{(1)}(t)+[\CE_\infty^{(1)}(t)]^2\right\}.
\end{equation}
Then, from the completely same proof as for \eqref{TE.p06}, one can
obtain
\begin{equation*}
\sup_{t\geq 0}\CE_\infty^{(1)}(t)\leq C\eps_0,
\end{equation*}
under the smallness condition of  $\eps_0$ and $\eps_{0,\nu}$. This
proves  \eqref{prop.main.2} by the definition of
$\CE_\infty^{(1)}(t)$ in \eqref{TE.def.e01}.

\subsection{Time decay of high-order energy}

At this time, in order to complete the proof of Proposition
\ref{prop.main}, it suffices to prove \eqref{prop.main.3} in
Proposition \ref{prop.main}, which gives the time-decay rates for
the high-order energy. First of all, let us obtain some
Lyapunov-type inequalities of the high-order energy on the basis of
Lemma \ref{lem.ES}, Lemma \ref{lem.ES.w} and Lemma \ref{a.pro.pri}
in the following two lemmas.

\begin{lemma}\label{lem.HES}
There is a high-order energy functional $\CE^{h}(u(t))$ defined by
\begin{equation}\label{lem.HES.0}
\CE^{h}(u(t))\sim \|\{\FI-\FP_1\}u(t)\|_{H^N}^2+\|\na_x \FP_1
u(t)\|_{L^2_\xi(H^{N-1}_x)}^2,
\end{equation}
such that
\begin{equation}\label{lem.HES.1}
    \frac{d}{dt}\CE^{h}(u(t))+\la \CD(u(t))\leq C\|\na_x\FP u(t)\|^2
\end{equation}
holds for any $t\geq 0$, where $\CD(u(t))$ is defined in
\eqref{a.prop.exi.2}.
\end{lemma}

\begin{proof}
Define
\begin{eqnarray}
 \CE^{h}(u(t))  &=& \|\{\FI-\FP\} u(t)\|^2+ \sum_{1\leq |\al|\leq
 N}(\|\pa_x^\al u(t)\|^2+\|\pa_x^\al \na_x\Phi(t)\|^2) \nonumber\\
 &&+\kappa_3 \CE_{x,\xi}(u(t))+\kappa_4 \CE_{free}(u(t)),\label{lem.HES.p1}
\end{eqnarray}
where $\CE_{x,\xi}(u(t))$ and $\CE_{free}(u(t))$ are given by
\eqref{a.pro.pri.3def} and \eqref{a.pro.pri.4}, respectively, and
constants $\kappa_3$ and $\kappa_4$ to be determined later satisfies
\begin{equation}\label{lem.HES.p2}
0<\kappa_3\ll \kappa_4\ll 1.
\end{equation}
Notice from the definition \eqref{a.pro.pri.4} of $\CE_{free}(u(t))$
that
\begin{eqnarray*}
|\CE_{free}(u(t))|&\leq& C\sum_{|\al|\leq
N-1}(\|\pa_x^\al\{\FI-\FP\}
u\|^2+\|\pa_x^\al \FP_0 u\|^2+\|\pa_x^\al \na_x b^{u}\|^2)\\
&\leq &  C \|\{\FI-\FP\} u(t)\|^2+ C\sum_{1\leq |\al|\leq
 N}(\|\pa_x^\al u(t)\|^2+\|\pa_x^\al \na_x\Phi(t)\|^2),
\end{eqnarray*}
which together with \eqref{lem.HES.p1} and \eqref{a.pro.pri.3def}
imply
\begin{eqnarray*}
\CE^{h}(u(t))&\sim&  \|\{\FI-\FP\} u(t)\|^2+ \sum_{1\leq |\al|\leq
 N}(\|\pa_x^\al u(t)\|^2+\|\pa_x^\al
 \na_x\Phi(t)\|^2)+\kappa_3\CE_{x,\xi}(u(t))\\
 &\sim & \|\{\FI-\FP\} u(t)\|_{H^N}^2+ \|\FP_0 u(t)\|_{L^2_\xi(H^{N-1}_x)}^2+\|\na_x \FP_1
u(t)\|_{L^2_\xi(H^{N-1}_x)}^2\\
&\sim & \|\{\FI-\FP_1\}u(t)\|_{H^N}^2+\|\na_x \FP_1
u(t)\|_{L^2_\xi(H^{N-1}_x)}^2,
\end{eqnarray*}
by taking $\kappa_4>0$ small enough and also  letting $\kappa_3>0$.
Thus, \eqref{lem.HES.0} holds true. Moreover, under the condition
\eqref{lem.HES.p2}, the linear combination of \eqref{lem.ES.0},
\eqref{a.pro.pri.2}, \eqref{a.pro.pri.3} and \eqref{a.pro.pri.5} in
the special case when $\de_\phi=0$ from the assumption
\eqref{ass.backg} yield
\begin{equation*}
    \frac{d}{dt}\CE^{h}(u(t))+\la \CD(u(t))\leq C\|\na_x\FP u(t)\|^2
    +C (\CE(u(t))+\sqrt{\CE(u(t))})\CD(u(t)),
\end{equation*}
where $\CD(u(t))$ is given by \eqref{a.prop.exi.2}. Therefore,
\eqref{lem.HES.1} follows from the above inequality and smallness of
$\CE(u(t))$ as mentioned at the beginning of Subsection
\ref{sec.ES.sub1}. This completes the proof of Lemma \ref{lem.HES}.
\end{proof}

\begin{lemma}\label{lem.wHES}
There is a weighted high-order energy functional $\CE^{h}_w(u(t))$
and a corresponding dissipation rate $\CD_w(u(t))$, which are
defined by
\begin{eqnarray}
\CE^{h}_w(u(t))&\sim& \|\{\FI-\FP_1\}u(t)\|_{H^N_w}^2+\|\na_x \FP_1
u(t)\|_{L^2_\xi(H^{N-1}_x)}^2,\label{lem.wHES.1}\\
\CD_w(u(t)) &\sim&  \|\{\FI-\FP_1\}u(t)\|_{H^N_{w^2}}^2 + \|\na_x
\FP_1 u(t)\|_{L^2_\xi(H^{N-1}_x)}^2,\label{lem.wHES.2}
\end{eqnarray}
such that
\begin{equation}\label{lem.wHES.3}
    \frac{d}{dt}\CE^{h}_w(u(t))+\la \CD_w(u(t))\leq C\|\na_x\FP u(t)\|^2
\end{equation}
holds for any $t\geq 0$.

\end{lemma}

\begin{proof}
Define
\begin{eqnarray*}
\CE^{h}_w(u(t)) &=& \|w^{1/2}\{\FI-\FP\}u(t)\|^2+\sum_{1\leq
|\al|\leq N}\|w^{1/2}\pa_x^\al u(t)\|^2\\
&&+\sum_{\substack{|\be|\geq 1\\
|\al|+|\be|\leq N }}\|w^{1/2}\pa_x^\al\pa_\xi^\be \{\FI-\FP\}
u\|^2+\kappa_5 \CE^{h}(u(t)),
\end{eqnarray*}
where $\CE^{h}(u(t))$ is given by \eqref{lem.HES.0} and $\kappa_5>0$
is a constant to be determined later. By using \eqref{lem.HES.0}, it
is straightforward to check
\begin{eqnarray*}
\CE^{h}_w(u(t)) &\sim& \|\{\FI-\FP\}u(t)\|_{H^N_{w^2}}^2+\|\na_x \FP
u(t)\|_{L^2_\xi(H^{N-1}_x)}^2+\kappa_5 \CE^{h}(u(t))\\
 &\sim& \|\{\FI-\FP\}u(t)\|_{H^N_{w^2}}^2+\|\na_x \FP
u(t)\|_{L^2_\xi(H^{N-1}_x)}^2+\|\FP_0 u(t)\|_{H^N}^2\\
&\sim& \|\{\FI-\FP_1\}u(t)\|_{H^N_{w^2}}^2+\|\na_x \FP_1
u(t)\|_{L^2_\xi(H^{N-1}_x)}^2,
\end{eqnarray*}
which implies \eqref{lem.wHES.1}. The rest is to verify
\eqref{lem.wHES.3}. Actually, by taking $\kappa_5>0$ large enough,
the linear combination of \eqref{lem.ES.w.1}, \eqref{lem.ES.w.2},
\eqref{lem.ES.w.3} and \eqref{lem.HES.1} yields
\begin{equation}\label{lem.wHES.p1}
    \frac{d}{dt}\CE^{h}_w(u(t))+\la \CD_w(u(t))\leq C\|\na_x\FP
    u(t)\|^2+ C \CE(u(t)) \|\{\FI-\FP\}u(t)\|_{H^N_{w^2}}^2,
\end{equation}
where $ \CD_w(u(t))$ takes the form
\begin{eqnarray}
 \CD_w(u(t))  &=& \|w \{\FI-\FP\} u(t)\|^2+\sum_{1\leq |\al|\leq N }\|w \pa_x^\al
 u(t)\|^2 \nonumber\\
 &&+\sum_{\substack{|\be|\geq 1\\
|\al|+|\be|\leq N }}\|w \pa_x^\al\pa_\xi^\be \{\FI-\FP\}
u(t)\|^2+\CD(u(t)).\label{lem.wHES.p2}
\end{eqnarray}
From the definition \eqref{a.prop.exi.2} of $\CD(u(t))$,
\eqref{lem.wHES.p2} implies  that \eqref{lem.wHES.2} holds true.
Since for the second term on the r.h.s. of \eqref{lem.wHES.p1} it
holds that
\begin{equation*}
 \CE(u(t)) \|\{\FI-\FP\}u(t)\|_{H^N_{w^2}}^2\leq
 C\CE(u(t))\CD_w(u(t)),
\end{equation*}
then \eqref{lem.wHES.3} follows  from \eqref{lem.wHES.p1} and
smallness of $\CE(u(t))$ again due to \eqref{smallness} and
smallness of $\CE(u_0)$. This completes the proof of Lemma
\ref{lem.wHES}.
\end{proof}

\medskip

Now, we are in a position to prove \eqref{prop.main.3}. For this, as
before, define a temporal function by
\begin{equation}\label{wHES.def.infty}
\CE^h_{w,\infty}(t)=\sup_{0\leq s\leq
t}(1+s)^{2(\frac{3}{4}-\eps)}\CE^h_w(u(s)),
\end{equation}
where $0<\eps\leq 3/4$ is an arbitrary constant, and $\CE^h_w(u)$ is
given by \eqref{lem.wHES.1}. For simplicity, also denote a constant
$\eps_1$ depending only on initial data by
\begin{equation}\label{wHES.def.eps1}
    \eps_1=\|u_0\|_{H^N_{w}}^2+\|u_0\|_{Z_1}^2.
\end{equation}
Notice $\CE^h_w(u_0)\leq C\eps_1$ which will be used later.

\medskip

\noindent{\bf {Proof of \eqref{prop.main.3} in Proposition
\ref{prop.main}}:}\ We begin with the Lyapunov-type inequality
\eqref{lem.wHES.3} for the weighted high-order energy functional
$\CE^h_w(u(t))$. Since
\begin{equation*}
  \CE^h_w(u(t))\leq C \CD_w(u(t))
\end{equation*}
holds by definitions \eqref{lem.wHES.1}-\eqref{lem.wHES.2} of
$\CE^h_w(u(t))$ and $\CD_w(u(t))$, \eqref{lem.wHES.3} implies
\begin{equation*}
    \frac{d}{dt}\CE^{h}_w(u(t))+\la \CE^{h}_w(u(t))\leq C\|\na_x\FP
    u(t)\|^2.
\end{equation*}
From the Gronwall inequality, it follows that
\begin{equation}\label{wHES.p1}
\CE^{h}_w(u(t))\leq \CE^{h}_w(u_0) e^{-\la t} +C \int_0^t e^{-\la
(t-s)}\|\na_x\FP u(s)\|^2ds.
\end{equation}
Notice that \eqref{wHES.p1} together with \eqref{prop.main.1} imply
that $\CE^{h}_w(u(t))$ has at least the same time-decay rate with
the total energy $\CE(u(t))$, that is,
\begin{equation}\label{wHES.p2}
\sup_{t\geq 0}(1+t)^{\frac{1}{2}}\CE^{h}_w(u(t))\leq C \eps_1 ,
\end{equation}
if conditions of \eqref{prop.main.1} hold and $\|u_0\|_{H^N_w}$ is
bounded. In what follows, we shall improve \eqref{wHES.p2} up to the
almost optimal rate in the sense that for any $0<\eps\leq 3/4$,
there is $\eta>0$ depending only on $\eps$ such that whenever
$\eps_0\leq \eta^2$ with $\eps_0$ defined in \eqref{TE.def.eps}, one
has
\begin{equation}\label{wHES.p3}
\sup_{t\geq 0}(1+t)^{2(\frac{3}{4}-\eps)}\CE^{h}_w(u(t))\leq C
\eps_1 ,
\end{equation}
which implies \eqref{prop.main.3}. In fact, one can obtain a formal
time-decay estimate on the first-order energy $\|\na_x\FP u(t)\|^2$
in terms of $\CE^h_{w,\infty}(t)$ in the same way as in
\eqref{TE.p05} or \eqref{TE.p07}. Firstly, similar to get
\eqref{TE.p04}, by applying Theorem \ref{thm.liDe} to \eqref{TE.p02}
with the help of the decomposition \eqref{TE.p03} for the source
term $G$, it follows that
\begin{eqnarray}
&\dis \|\na_x \FP u(t)\|^2
\leq C(1+t)^{-\frac{3}{2} } (\|u_0\|_{Z_1}^2+\|\na_x u_0\|^2) \nonumber\\
&\dis +\sum_{j=1}^2 C\int_0^t
(1+t-s)^{-\frac{3}{2}}(\|\nu^{-1/2}\na_x\{\FI-\FP\}
G_j(s)\|_{Z_1}^2+
\|\nu^{-1/2}\na_x\{\FI-\FP\} G_j(s)\|^2) ds \nonumber\\
&\dis +C\left[\int_0^t (1+t-s)^{-\frac{3}{4}}(\|\na_x \FP_1 G_2
(s)\|_{Z_1}+\|\na_x \FP_1 G_2(s)\|)ds\right]^2,\label{wHES.p4}
\end{eqnarray}
where $G_1$ and $G_2$ are defined in \eqref{TE.def.G12}. Next, it is
straightforward to check
\begin{eqnarray*}
&\dis \|\nu^{-1/2}\na_x\{\FI-\FP\} G_j(t)\|_{Z_1}^2+
\|\nu^{-1/2}\na_x\{\FI-\FP\} G_j(t)\|^2\\
&\dis \leq C \CE_w^h(u(t))[\CE_w^h(u(t))+\CE(u(t))]
\end{eqnarray*}
for $j=1,2$, and
\begin{equation*}
  \|\na_x \FP_1 G_2
(t)\|_{Z_1}+\|\na_x \FP_1 G_2(t)\|\leq C [
\CE_w^h(u(t))\CE(u(t))]^{\frac{1}{2}}.
\end{equation*}
Plugging the above inequalities into \eqref{wHES.p4} gives
\begin{eqnarray}
 \|\na_x \FP u(t)\|^2
&\leq & C(1+t)^{-\frac{3}{2} } (\|u_0\|_{Z_1}^2+\|\na_x u_0\|^2) \nonumber\\
&&\dis + C \int_0^t
(1+t-s)^{-\frac{3}{2}}\CE_w^h(u(s))[\CE_w^h(u(s))+\CE(u(s))]ds \nonumber\\
&&\dis +C\left[\int_0^t (1+t-s)^{-\frac{3}{4}}[
\CE_w^h(u(s))\CE(u(s))]^{\frac{1}{2}}ds\right]^2. \label{wHES.p5}
\end{eqnarray}
Notice that from the time integrations of \eqref{TE.p00} and
\eqref{lem.HES.1} as well as the definition \eqref{wHES.def.eps1} of
$\eps_1$, it holds that
\begin{equation*}
\sup_{t\geq 0}[\CE_w^h(u(t))+\CE(u(t))]\leq C \eps_1,
\end{equation*}
and also recall from \eqref{TE.p06}, \eqref{TE.def.e01} and
\eqref{TE.def.eps} that
\begin{equation*}
    \CE(u(t))\leq C \eps_0 (1+t)^{-\frac{1}{2}},
\end{equation*}
where $\eps_0$ is defined in \eqref{TE.def.eps}. Then, it follows
from \eqref{wHES.p5} that
\begin{eqnarray*}
 \|\na_x \FP u(t)\|^2
&\leq & C(1+t)^{-\frac{3}{2} } (\|u_0\|_{Z_1}^2+\|\na_x u_0\|^2)\\
&&\dis + C \eps_1\int_0^t
(1+t-s)^{-\frac{3}{2}}\CE_w^h(u(s))ds\\
&&\dis +C\eps_0\left[\int_0^t
(1+t-s)^{-\frac{3}{4}}(1+s)^{-\frac{1}{4}}[
\CE_w^h(u(s))]^{\frac{1}{2}}ds\right]^2,
\end{eqnarray*}
which further from the definition \eqref{wHES.def.infty} of
$\CE^h_{w,\infty}(t)$ implies
\begin{eqnarray}
 \|\na_x \FP u(t)\|^2
&\leq & C(1+t)^{-\frac{3}{2} } (\|u_0\|_{Z_1}^2+\|\na_x u_0\|^2) \nonumber\\
&&\dis + C \eps_1\CE^h_{w,\infty}(t)\int_0^t
(1+t-s)^{-\frac{3}{2}}(1+s)^{-2(\frac{3}{4}-\eps)}ds \nonumber\\
&&\dis +C \eps_0\CE^h_{w,\infty}(t) \left[\int_0^t
(1+t-s)^{-\frac{3}{4}}(1+s)^{-\frac{1}{4}-(\frac{3}{4}-\eps)}ds\right]^2.
 \label{wHES.p6}
\end{eqnarray}
Notice that for $0<\eps\leq 3/4$,
\begin{eqnarray*}
&& \int_0^t (1+t-s)^{-\frac{3}{2}}(1+s)^{-2(\frac{3}{4}-\eps)}ds\leq
C(1+t)^{-2(\frac{3}{4}-\eps)},\\
&&\int_0^t
(1+t-s)^{-\frac{3}{4}}(1+s)^{-\frac{1}{4}-(\frac{3}{4}-\eps)}ds\leq
C_\eps (1+t)^{-(\frac{3}{4}-\eps)},
\end{eqnarray*}
where $C>0$ in the first inequality can be taken uniformly in $0
<\eps\leq 3/4$, while $C_\eps>0$ in the second inequality has to
tend to infinity as $\eps$ goes to zero. Then, it follows from
\eqref{wHES.p6} that
\begin{equation*}
\|\na_x \FP u(t)\|^2 \leq  (1+t)^{-(\frac{3}{2}-2\eps) }[ C\eps_0 +
(C\eps_1+C_\eps\eps_0)\CE^h_{w,\infty}(t)].
\end{equation*}
Combined with the  above estimate, \eqref{wHES.p1} gives
\begin{equation*}
\CE^{h}_w(u(t))\leq  (1+t)^{-(\frac{3}{2}-2\eps) }[ C\eps_1 +
(C\eps_1+C_\eps\eps_0)\CE^h_{w,\infty}(t)],
\end{equation*}
for any $t\geq 0$, where $\max\{\CE^{h}_w(u_0),\eps_0\}\leq C
\eps_1$ was used. Thus, it follows that
\begin{equation*}
\CE^h_{w,\infty}(t)\leq  C\eps_1 +
(C\eps_1+C_\eps\eps_0)\CE^h_{w,\infty}(t),
\end{equation*}
that is,
\begin{equation*}
\CE^h_{w,\infty}(t)\leq  C\eps_1 + C_\eps\eps_0\CE^h_{w,\infty}(t),
\end{equation*}
since $\eps_1$ is small enough. Therefore, given any $0<\eps\leq
3/4$, one can choose $\eta=1/\sqrt{2C_\eps}$ so that whenever
$\eps_0\leq \eta^2$, it holds that
\begin{equation*}
\CE^h_{w,\infty}(t)\leq  C\eps_1,
\end{equation*}
for any $t\geq 0$, which proves \eqref{wHES.p3} and hence
\eqref{prop.main.3} by the definition \eqref{wHES.def.infty} of
$\CE^h_{w,\infty}(t)$. This also completes the proof of Proposition
\ref{prop.main}.

\section{Appendix}\label{sec.app}

In this appendix, we  shall state the existence of the stationary
solution and its nonlinear stability for the VPB system
\eqref{VPB.1}-\eqref{VPB.2}. For that, let us define the weighted
norm $\|\cdot\|_{W^{m,\infty}_\theta}$ by
\begin{equation*}
    \|g\|_{W^{m,\infty}_\theta}=\sup_{x\in\R^3} (1+|x|)^\theta \sum_{|\al|\leq
    m} |\pa_x^\al g(x)|,
\end{equation*}
for suitable $g=g(x)$ and for an integer $m\geq 0$ and $\theta\geq
0$. \cite{DY-09VPB} proved

\begin{proposition}[existence of stationary solutions]\label{prop.ss}
Let the integer $m\geq 0$ and $\theta\geq 0$.  Suppose that
$\|\bar{\rho}-1\|_{W^{m,\infty}_\theta}$ is small enough. Then the
following elliptic equation with exponential nonlinearity:
\begin{equation*}
    \De_x\phi=e^\phi-\bar{\rho}(x),
\end{equation*}
admits a unique solution $\phi=\phi(x)$ satisfying
\begin{equation*}
\|\phi\|_{W^{m,\infty}_\theta}\leq
C\|\bar{\rho}-1\|_{W^{m,\infty}_\theta},
\end{equation*}
for some constant $C$.

\end{proposition}

From Proposition \ref{prop.ss} above, it is straightforward to check
that the VPB system \eqref{VPB.1}-\eqref{VPB.2} has a stationary
solution $(f_\ast,\Phi_\ast)$ given by $f_\ast=e^\phi\FM$,
$\Phi_\ast=\phi$.  To state the stability of the stationary state
$(f_\ast,\Phi_\ast)$, set the perturbation $u=u(t,x,\xi)$ by
\begin{equation*}
    f=e^\phi\FM+\sqrt{\FM}u.
\end{equation*}
Then $u$ satisfies the perturbed system:
\begin{equation}\label{A.CP}
    \left\{\begin{array}{l}
     \dis \pa_t u+\xi\cdot\na_x u+\na_x(\Phi+\phi)\cdot\na_\xi
u-\frac{1}{2}\xi\cdot\na_x (\Phi+\phi) u -\xi\cdot\na_x\Phi e^\phi
\sqrt{\FM}\\[3mm]
\dis\hspace{4cm}=e^\phi\FL u+\Ga(u,u),\\[3mm]
\dis \Phi=-\frac{1}{4\pi|x|}\ast \int_{\R^3} \sqrt{\FM}ud\xi,\ \
(t,x,\xi)\in (0,\infty)\times\R^3\times\R^3,
    \end{array}\right.
\end{equation}
with given initial data
\begin{equation}\label{A.CP.ID}
    u(0,x,\xi)=u_0(x,\xi)\equiv \frac{f_0-e^\phi\FM}{\sqrt{\FM}},\ \
    (x,\xi)\in\R^3\times\R^3.
\end{equation}
Here $\FL u$ and $\Ga(u,u)$ are defined by in the same forms as in
\eqref{def.L} and \eqref{def.Ga}, respectively. \cite{DY-09VPB} also
proved

\begin{proposition}[stability of stationary solutions]\label{a.prop.exi}
Let $N\geq 4$. Suppose that $\|\bar{\rho}-1\|_{W^{N+1,\infty}_2}$ is
small enough. Then, there are the equivalent energy functional
$\CE(\cdot)$ and the corresponding energy dissipation rate
$\CD(\cdot)$ defined by
\begin{eqnarray}
\CE(u(t))&\sim & \|u\|_{H^N}^2
  +\| \na_x \De_x^{-1}\FP_0u\|^2,\label{a.prop.exi.1}\\
\CD(u(t))&\sim &  \|\{\FI-\FP_1\}u\|_{H^N_\nu}^2 + \|\na_x \FP_1
u\|_{L^2_\xi(H^{N-1}_x)}^2,\label{a.prop.exi.2}
\end{eqnarray}
such that the following holds. If $f_0=e^\phi\FM+\sqrt{\FM}u_0\geq
0$ and $\CE(u_0)$ is sufficiently small, then  the Cauchy problem
\eqref{A.CP}-\eqref{A.CP.ID} of the VPB system admits a unique
global solution $u(t,x,\xi)$ satisfying $f(t,x,\xi)\equiv
e^\phi\FM+\sqrt{\FM}u(t,x,\xi)\geq 0$, and
\begin{equation}\label{a.prop.exi.3}
    \frac{d}{dt}\CE(u(t))+\la \CD(u(t))\leq 0.
\end{equation}

\end{proposition}

The following lemma was obtained by \cite{DY-09VPB} in the proof of
uniform {\it a priori} estimates for  the global-in-time stability
of the stationary solution.

\begin{lemma}[{\it a priori} estimates]\label{a.pro.pri}
Let all conditions  of Proposition \ref{a.prop.exi} hold and let $u$
be the corresponding solution, and $(a^u,b^u,c^u)$ be defined in
\eqref{def.a}-\eqref{def.c}. Denote
$\de_\phi=\|\phi\|_{W^{N+1,\infty}_2}$. Then, the following uniform
{\it a priori} estimates hold for any $t\geq 0$:

\medskip

\noindent (i) zero-order:
\begin{eqnarray}
 &&\dis  \frac{d}{dt}\left(\|u\|^2+\|\na_x \Phi\|^2-2 \int_{\R^3} e^{-\phi}|b^u|^2 c^u\, dx\right)
 +\la \iint_{\R^3\times\R^3}
  \nu(\xi)|\{\FI-\FP\}u|^2dxd\xi\nonumber\\
 &&\dis  \leq C(\de_\phi+\sqrt{\CE(u(t))})\CD(u(t));\label{a.pro.pri.1}
\end{eqnarray}

\medskip

\noindent (ii) spatial derivatives:
\begin{eqnarray}
&&\frac{d}{dt}\sum_{1\leq |\al|\leq N}\left(\|\pa_x^\al u\|^2+
\|\pa_x^\al\na_x\Phi\|^2\right)+\la\sum_{1\leq|\al|\leq N}
\iint_{\R^3\times\R^3} \nu(\xi) |\pa_x^\al\{\FI-\FP\}u|^2dxd\xi\nonumber\\
&&\leq C(\de_\phi+\sqrt{\CE(u(t))})\CD(u(t));\label{a.pro.pri.2}
\end{eqnarray}

\medskip

\noindent (iii) mixed spatial-velocity derivatives:
\begin{eqnarray}
&&\dis \frac{d}{dt}\CE_{x,\xi}(u(t))+\la\sum_{\substack{|\be|\geq 1 \nonumber\\
|\al|+|\be|\leq N}}\iint_{\R^3\times\R^3} \nu(\xi) |\pa_x^\al\pa_\xi^\be\{\FI-\FP\}u|^2dxd\xi\nonumber\\
&&\leq C(\de_\phi+\sqrt{\CE(u(t))})\CD(u(t))\nonumber\\
&&\dis\ \ \ +C\sum\limits_{ |\al|\leq
N-1}\|\pa_x^\al\na_x(a^u,b^u,c^u)\|^2+C\sum_{|\al|\leq
N}\|\nu^{1/2}\pa_x^\al\{\FI-\FP\} u\|^2,\label{a.pro.pri.3}
\end{eqnarray}
where
\begin{equation}\label{a.pro.pri.3def}
\CE_{x,\xi}(u(t))=\sum_{k=1}^NC_{N,k}\sum_{\substack{|\be|=k\\
|\al|+|\be|\leq N }}\|\pa_x^\al\pa_\xi^\be \{\FI-\FP\} u\|^2,
\end{equation}
for some proper positive constants $C_{N,k}$;

\medskip

\noindent (iv) macroscopic dissipation: There is a temporal free
energy functional $\CE_{free}(u(t))$ in the form of
\begin{eqnarray}
 \CE_{free}(u(t)) &=& \kappa_0\sum_{|\al|\leq N-1}\sum_{ij}\langle
A_{ij}(\pa_x^\al \{\FI-\FP\}u), \pa_x^\al (\pa_ib_j^u+\pa_j
b_i^u)\rangle \nonumber\\
&&+ \kappa_0\sum_{|\al|\leq N-1}\sum_{j}\langle A_{jj}(\pa_x^\al
\{\FI-\FP\}u), \pa_x^\al\pa_jb_j^u\rangle  \nonumber\\
&&+\kappa_0\sum_{|\al|\leq N-1}\sum_{i}\langle B_i(\pa_x^\al
\{\FI-\FP\}u),\pa_x^\al
\pa_i c^u\rangle \nonumber\\
&&-\sum_{|\al|\leq N-1}\langle \pa_x^\al  (a^u+3c^u),\pa_x^\al
\na_x\cdot b^u\rangle \label{a.pro.pri.4}
\end{eqnarray}
for some constant $\kappa_0>0$, such that one has
\begin{eqnarray}
\frac{d}{dt} \CE_{free}(u(t))&+&\la \CD_{mac}(u(t))\leq
C\sum_{|\al|\leq N}\|\pa_x^\al
\{\FI-\FP\}u\|^2+\CE(u(t))\CD(u(t)),\label{a.pro.pri.5}
\end{eqnarray}
where $\CD_{mac}(u(t))$ is the macroscopic dissipation rate given by
\begin{equation*}
\CD_{mac}(u(t))=\sum_{|\al|\leq
N-1}\|\pa_x^\al\na_x(a^u+3c^u,b^u,c^u)\|^2+\| a^u+3c^u\|^2.
\end{equation*}

\end{lemma}

\medskip

\noindent {\bf Acknowledgment:}\,\, Renjun Duan
 would like to thank the financial support provided by RICAM.
The research of Robert Strain has been partly supported by the NSF,
in particular by NSF fellowship DMS-0602513 and NSF grant
DMS-0901463.

\end{document}